\documentclass[12pt]{amsart}       
\usepackage{txfonts}
\usepackage{amssymb}
\usepackage{eucal}
\usepackage{graphicx}
\usepackage{amsmath}
\usepackage{amscd}
\usepackage[all]{xy}           
\usepackage{amsfonts,latexsym}
\usepackage{xspace}
\usepackage{epsfig}
\usepackage{float}
\usepackage{color}
\usepackage{fancybox}
\usepackage{colordvi}
\usepackage{multicol}
\usepackage{colordvi}
\usepackage[colorlinks,final,backref=page,hyperindex,hypertex]{hyperref}
\usepackage[active]{srcltx} 
\usepackage{tikz}

\usepackage{graphicx}

\newtheorem{theorem}{Theorem}[section]
\newtheorem{prop}[theorem]{Proposition}
\theoremstyle{definition}
\newtheorem{defn}[theorem]{Definition}
\newtheorem{lemma}[theorem]{Lemma}
\newtheorem{coro}[theorem]{Corollary}
\newtheorem{prop-def}{Proposition-Definition}[section]
\newtheorem{coro-def}{Corollary-Definition}[section]

\newtheorem{remark}[theorem]{Remark}

\newtheorem{exam}[theorem]{Example}

\newcommand{\nc}{\newcommand}
\nc{\tred}[1]{\textcolor{red}{#1}}
\nc{\tblue}[1]{\textcolor{blue}{#1}}
\nc{\tgreen}[1]{\textcolor{green}{#1}}
\nc{\tpurple}[1]{\textcolor{purple}{#1}}
\nc{\btred}[1]{\textcolor{red}{\bf #1}}
\nc{\btblue}[1]{\textcolor{blue}{\bf #1}}
\nc{\btgreen}[1]{\textcolor{green}{\bf #1}}
\nc{\btpurple}[1]{\textcolor{purple}{\bf #1}}
\nc{\NN}{{\mathbb N}}
\nc{\ncsha}{{\mbox{\cyr X}^{\mathrm NC}}} \nc{\ncshao}{{\mbox{\cyr
X}^{\mathrm NC}_0}}

\newcommand{\efootnote}[1]{}

\renewcommand{\textbf}[1]{}

\newcommand{\delete}[1]{}

\nc{\mlabel}[1]{\label{#1}}  
\nc{\mcite}[1]{\cite{#1}}  
\nc{\mref}[1]{\ref{#1}}  
\nc{\mbibitem}[1]{\bibitem{#1}} 

\delete{
\nc{\mlabel}[1]{\label{#1}{\hfill \hspace{1cm}{\bf{{\ }\hfill(#1)}}}}
\nc{\mcite}[1]{\cite{#1}{{\bf{{\ }(#1)}}}}  
\nc{\mref}[1]{\ref{#1}{{\bf{{\ }(#1)}}}}  
\nc{\mbibitem}[1]{\bibitem[\bf #1]{#1}} 
}

\nc{\opa}{\ast} \nc{\opb}{\odot} \nc{\op}{\bullet} \nc{\pa}{\frakL}
\nc{\arr}{\rightarrow} \nc{\lu}[1]{(#1)} \nc{\mult}{\mrm{mult}}
\nc{\diff}{\mathfrak{Diff}}
\nc{\opc}{\sharp}\nc{\opd}{\natural}
\nc{\ope}{\circ}
\nc{\dpt}{\mathrm{d}}
\nc{\hck}{H_{RT}}
\nc{\vdf}{\calf}
\nc{\ldf}{\calf_\ell}
\nc{\hlf}{H_\ell}
\nc{\onek}{\mathbf{1}_\bfk}
\nc{\diam}{alternating\xspace}
\nc{\Diam}{Alternating\xspace}
\nc{\cdiam}{canonical alternating\xspace}
\nc{\Cdiam}{Canonical alternating\xspace}
\nc{\AW}{\mathcal{A}}

\nc{\ari}{\mathrm{ar}}

\nc{\lef}{\mathrm{lef}}

\nc{\Sh}{\mathrm{ST}}

\nc{\Cr}{\mathrm{Cr}}

\nc{\st}{{Schr\"oder tree}\xspace}
\nc{\sts}{{Schr\"oder trees}\xspace}

\nc{\vertset}{\Omega} 

\nc{\assop}{\quad \begin{picture}(5,5)(0,0)
\line(-1,1){10}
\put(-2.2,-2.2){$\bullet$}
\line(0,-1){10}\line(1,1){10}
\end{picture} \quad \smallskip}

\nc{\operator}{\begin{picture}(5,5)(0,0)
\line(0,-1){6}
\put(-2.6,-1.8){$\bullet$}
\line(0,1){9}
\end{picture}}

\nc{\idx}{\begin{picture}(6,6)(-3,-3)
\put(0,0){\line(0,1){6}}
\put(0,0){\line(0,-1){6}}
\end{picture}}

\nc{\pb}{{\mathrm{pb}}}
\nc{\Lf}{{\mathrm{Lf}}}

\nc{\lft}{{left tree}\xspace}
\nc{\lfts}{{left trees}\xspace}

\nc{\fat}{{fundamental averaging tree}\xspace}

\nc{\fats}{{fundamental averaging trees}\xspace}
\nc{\avt}{\mathrm{Avt}}

\nc{\rass}{{\mathit{RAss}}}

\nc{\aass}{{\mathit{AAss}}}

\nc{\vin}{{\mathrm Vin}}    
\nc{\lin}{{\mathrm Lin}}    
\nc{\inv}{\mathrm{I}n}
\nc{\gensp}{V} 
\nc{\genbas}{\mathcal{V}} 
\nc{\bvp}{V_P}     
\nc{\gop}{{\,\omega\,}}     

\nc{\bin}[2]{ (_{\stackrel{\scs{#1}}{\scs{#2}}})}  
\nc{\binc}[2]{ \left (\!\! \begin{array}{c} \scs{#1}\\
    \scs{#2} \end{array}\!\! \right )}  
\nc{\bincc}[2]{  \left ( {\scs{#1} \atop
    \vspace{-1cm}\scs{#2}} \right )}  
\nc{\bs}{\bar{S}} \nc{\cosum}{\sqsubset} \nc{\la}{\longrightarrow}
\nc{\rar}{\rightarrow} \nc{\dar}{\downarrow} \nc{\dprod}{**}
\nc{\dap}[1]{\downarrow \rlap{$\scriptstyle{#1}$}}
\nc{\md}{\mathrm{dth}} \nc{\uap}[1]{\uparrow
\rlap{$\scriptstyle{#1}$}} \nc{\defeq}{\stackrel{\rm def}{=}}
\nc{\disp}[1]{\displaystyle{#1}} \nc{\dotcup}{\
\displaystyle{\bigcup^\bullet}\ } \nc{\gzeta}{\bar{\zeta}}
\nc{\hcm}{\ \hat{,}\ } \nc{\hts}{\hat{\otimes}}
\nc{\barot}{{\otimes}} \nc{\free}[1]{\bar{#1}}
\nc{\uni}[1]{\tilde{#1}} \nc{\hcirc}{\hat{\circ}} \nc{\lleft}{[}
\nc{\lright}{]} \nc{\lc}{\lfloor} \nc{\rc}{\rfloor}
\nc{\curlyl}{\left \{ \begin{array}{c} {} \\ {} \end{array}
    \right .  \!\!\!\!\!\!\!}
\nc{\curlyr}{ \!\!\!\!\!\!\!
    \left . \begin{array}{c} {} \\ {} \end{array}
    \right \} }
\nc{\longmid}{\left | \begin{array}{c} {} \\ {} \end{array}
    \right . \!\!\!\!\!\!\!}
\nc{\onetree}{\bullet} \nc{\ora}[1]{\stackrel{#1}{\rar}}
\nc{\ola}[1]{\stackrel{#1}{\la}}
\nc{\ot}{\otimes} \nc{\mot}{{{\boxtimes\,}}}
\nc{\otm}{\overline{\boxtimes}} \nc{\sprod}{\bullet}
\nc{\scs}[1]{\scriptstyle{#1}} \nc{\mrm}[1]{{\rm #1}}
\nc{\margin}[1]{\marginpar{\rm #1}}   
\nc{\dirlim}{\displaystyle{\lim_{\longrightarrow}}\,}
\nc{\invlim}{\displaystyle{\lim_{\longleftarrow}}\,}
\nc{\mvp}{\vspace{0.3cm}} \nc{\tk}{^{(k)}} \nc{\tp}{^\prime}
\nc{\ttp}{^{\prime\prime}} \nc{\svp}{\vspace{2cm}}
\nc{\vp}{\vspace{8cm}} \nc{\proofbegin}{\noindent{\bf Proof: }}
\nc{\proofend}{$\blacksquare$ \vspace{0.3cm}}
\nc{\modg}[1]{\!<\!\!{#1}\!\!>}
\nc{\intg}[1]{F_C(#1)} \nc{\lmodg}{\!
<\!\!} \nc{\rmodg}{\!\!>\!}
\nc{\cpi}{\widehat{\Pi}}
\nc{\sha}{{\mbox{\cyr X}}}  
\nc{\shap}{{\mbox{\cyrs X}}} 
\nc{\shpr}{\diamond}    
\nc{\shp}{\ast} \nc{\shplus}{\shpr^+}
\nc{\shprc}{\shpr_c}    
\nc{\msh}{\ast} \nc{\zprod}{m_0} \nc{\oprod}{m_1}
\nc{\vep}{\epsilon} \nc{\labs}{\mid\!} \nc{\rabs}{\!\mid}
\nc{\sqmon}[1]{\langle #1\rangle}

\nc{\mmbox}[1]{\mbox{\ #1\ }} \nc{\dep}{\mrm{dep}} \nc{\fp}{\mrm{FP}}
\nc{\rchar}{\mrm{char}} \nc{\End}{\mrm{End}} \nc{\Fil}{\mrm{Fil}}
\nc{\Mor}{Mor\xspace} \nc{\gmzvs}{gMZV\xspace}
\nc{\gmzv}{gMZV\xspace} \nc{\mzv}{MZV\xspace}
\nc{\mzvs}{MZVs\xspace} \nc{\Hom}{\mrm{Hom}} \nc{\id}{\mrm{id}}
\nc{\im}{\mrm{im}} \nc{\incl}{\mrm{incl}} \nc{\map}{\mrm{Map}}
\nc{\mchar}{\rm char} \nc{\nz}{\rm NZ} \nc{\supp}{\mathrm Supp}

\nc{\Alg}{\mathbf{Alg}} \nc{\Bax}{\mathbf{Bax}} \nc{\bff}{\mathbf f}
\nc{\bfk}{{\bf k}} \nc{\bfone}{{\bf 1}} \nc{\bfx}{\mathbf x}
\nc{\bfy}{\mathbf y}
\nc{\base}[1]{\bfone^{\otimes ({#1}+1)}} 
\nc{\Cat}{\mathbf{Cat}}

\nc{\detail}{\marginpar{\bf More detail}
    \noindent{\bf Need more detail!}
    \svp}
\nc{\Int}{\mathbf{Int}} \nc{\Mon}{\mathbf{Mon}}
\nc{\rbtm}{{shuffle }} \nc{\rbto}{{Rota-Baxter }}
\nc{\remarks}{\noindent{\bf Remarks: }} \nc{\Rings}{\mathbf{Rings}}
\nc{\Sets}{\mathbf{Sets}} \nc{\wtot}{\widetilde{\odot}}
\nc{\wast}{\widetilde{\ast}} \nc{\bodot}{\bar{\odot}}
\nc{\bast}{\bar{\ast}} \nc{\hodot}[1]{\odot^{#1}}
\nc{\hast}[1]{\ast^{#1}} \nc{\mal}{\mathcal{O}}
\nc{\tet}{\tilde{\ast}} \nc{\teot}{\tilde{\odot}}
\nc{\oex}{\overline{x}} \nc{\oey}{\overline{y}}
\nc{\oez}{\overline{z}} \nc{\oef}{\overline{f}}
\nc{\oea}{\overline{a}} \nc{\oeb}{\overline{b}}
\nc{\weast}[1]{\widetilde{\ast}^{#1}}
\nc{\weodot}[1]{\widetilde{\odot}^{#1}} \nc{\hstar}[1]{\star^{#1}}
\nc{\lae}{\langle} \nc{\rae}{\rangle}
\nc{\lf}{\lfloor}
\nc{\rf}{\rfloor}

\nc{\QQ}{{\mathbb Q}}
\nc{\RR}{{\mathbb R}} \nc{\ZZ}{{\mathbb Z}}

\nc{\cala}{{\mathcal A}} \nc{\calb}{{\mathcal B}}
\nc{\calc}{{\mathcal C}}
\nc{\cald}{{\mathcal D}} \nc{\cale}{{\mathcal E}}
\nc{\calf}{{\mathcal F}} \nc{\calg}{{\mathcal G}}
\nc{\calh}{{\mathcal H}} \nc{\cali}{{\mathcal I}}
\nc{\call}{{\mathcal L}} \nc{\calm}{{\mathcal M}}
\nc{\caln}{{\mathcal N}} \nc{\calo}{{\mathcal O}}
\nc{\calp}{{\mathcal P}} \nc{\calr}{{\mathcal R}}
\nc{\cals}{{\mathcal S}} \nc{\calt}{{\mathcal T}}
\nc{\calu}{{\mathcal U}} \nc{\calw}{{\mathcal W}} \nc{\calk}{{\mathcal K}}
\nc{\calx}{{\mathcal X}} \nc{\CA}{\mathcal{A}}

\nc{\fraka}{{\mathfrak a}} \nc{\frakA}{{\mathfrak A}}
\nc{\frakb}{{\mathfrak b}} \nc{\frakB}{{\mathfrak B}}
\nc{\frakD}{{\mathfrak D}} \nc{\frakF}{\mathfrak{F}}
\nc{\frakf}{{\mathfrak f}} \nc{\frakg}{{\mathfrak g}}
\nc{\frakH}{{\mathfrak H}} \nc{\frakL}{{\mathfrak L}}
\nc{\frakM}{{\mathfrak M}} \nc{\bfrakM}{\overline{\frakM}}
\nc{\frakm}{{\mathfrak m}} \nc{\frakP}{{\mathfrak P}}
\nc{\frakN}{{\mathfrak N}} \nc{\frakp}{{\mathfrak p}}
\nc{\frakS}{{\mathfrak S}} \nc{\frakT}{\mathfrak{T}}
\nc{\frakX}{{\mathfrak X}}
\nc{\BS}{\mathbb{S
}}

\font\cyr=wncyr10 \font\cyrs=wncyr7
\nc{\li}[1]{\textcolor{red}{Li:#1}}
\nc{\yi}[1]{\textcolor{blue}{Yi: #1}}
\nc{\xing}[1]{\textcolor{purple}{Xing:#1}}
\nc{\revise}[1]{\textcolor{red}{#1}}

\nc{\ID}{{\rm I}}\nc{\lbar}[1]{\overline{#1}}\nc{\bre}{{\rm bre}}
\nc{\sd}{\cals}\nc{\rb}{\rm RB}\nc{\A}{\rm A}\nc{\LL}{\rm L}\nc{\tx}{\tilde{X}}
\nc{\col}{\Delta_{L}}\nc{\mul}{m_{\mathrm{RT}}}\nc{\ul}{u_{RT}}\nc{\epl}{\epsilon_{RT}}
\nc{\hl}{H_{RT}}\nc{\arro}[1]{#1}\nc{\px}{P_{\tx}}\nc{\pw}{P_{\mathfrak{w}}}\nc{\pl}{B^+}
\nc{\pp}{\pl}\nc{\ppp}[1]{B^+(#1)}\nc{\dw}{\diamond}\nc{\dl}{\diamond}
\nc{\ncshaw}{\sha^{{\rm NC}}_{\mathfrak{w}}}\nc{\ncshal}{\sha^{{\rm NC}}_{{\rm \ell}}}
\nc{\ver}{\rm V}\nc{\ld}{l}\nc{\del}{\Delta_{{\rm \ell}}}\nc{\epsl}{\epsilon_{{\rm \ell}}}
\nc{\uul}{u_{{\rm \ell}}}\nc{\oneh}{\mathbf{1}}\nc{\onew}{\mathbf{1}}
\nc{\etree}{1} \nc{\conc}{m_{RT}} \nc{\subq}{\bfk Q_l} \nc{\fid}{\unlhd}  \nc{\sfid}{\lhd}
\nc{\lhl}{\leq_{h,l}} \nc{\ghl}{\geq_{hl}}

\nc{\hrtb}{H_{RT}(X\sqcup\Omega)} \nc{\hrts}{H_{RT}(X, \Omega)}\nc{\rts}{\mathcal{T}(X, \Omega)}\nc{\rfs}{\mathcal{F}(X, \Omega)} \nc{\fm}{m_{FM}} \nc{\coll}{\Delta_\lambda}

\begin{document}

\title[Weighted infinitesimal unitary bialgebras, pre-Lie and matrix algebras]{Weighted infinitesimal unitary bialgebras on matrix algebras and weighted associative Yang-Baxter equations}
%
\author{Yi Zhang}
\address{Department of Mathematics, Lanzhou University, Lanzhou, Gansu 730000, P.\,R. China}
         \email{zhangy2016@lzu.edu.cn}

\author{Xing Gao$^{*}$}\footnotetext{* Corresponding author.}
\address{School of Mathematics and Statistics, Key Laboratory of Applied Mathematics and Complex Systems, Lanzhou University, Lanzhou, Gansu 730000, P.\,R. China}
         \email{gaoxing@lzu.edu.cn}


\author{Jia-Wen Zheng}
\address{Department of Mathematics, Lanzhou University, Lanzhou, Gansu 730000, P.\,R. China}
         \email{zhengjw16@lzu.edu.cn}

\date{\today}
\begin{abstract}
We equip a matrix algebra with a weighted infinitesimal unitary bialgebraic structure, via a construction of a suitable coproduct. Furthermore, an infinitesimal unitary Hopf algebra, under the view of Aguiar, is constructed on a matrix algebra.
By exploring the relationship between weighted infinitesimal bialgebras and pre-Lie algebras, we construct a pre-Lie algebraic structure and then a new Lie algebraic structure on a matrix algebra.
We also introduce the weighted associative Yang-Baxter equations (AYBEs) and obtain the relationship between solutions of  weighted AYBEs and weighted infinitesimal unitary bialgebras.
We give a bijection between the solutions of the associative Yang-Baxter equation of weight $\lambda$  and Rota-Baxter operators of weight $-\lambda$ on matrix algebras.
As a consequence, weighted quasitriangular infinitesimal unitary bialgebras are constructed, which generalize the results studied by Aguiar. Finally, We show that any weighted quasitriangular infinitesimal unitary bialgebra can be made into a dendriform algebra.
\end{abstract}

\subjclass[2010]{
15A30, 
16W99, 
16T10, 
17B60, 
17D25  	
}

\keywords{matrix algebra; infinitesimal bialgebra; pre-Lie algebra; associative Yang-Baxter equation}

\maketitle

\tableofcontents

\setcounter{section}{0}

\allowdisplaybreaks

\section{Introduction}
The interaction between studies in pure mathematics and mathematical
physics has long been a rich source of inspirations that benefited both
fields. This paper arose from an attempt to connect these two fields by establishing connections among weighted infinitesimal unitary bialgebras, weighted associative Yang-Baxter equations, pre-Lie algebras, Rota-Baxter algebras and dendriform algebras.

Infinitesimal bialgebras, introduced by Joni and Rota~\mcite{JR}, are in order to give an
algebraic framework for the calculus of Newton divided differences.
Namely, an infinitesimal bialgebra is a module $A$ which is simultaneously an algebra (possibly without a unit) and a coalgebra (possibly without a counit) such that the coproduct $\Delta$ is a derivation of $A$ in the sense:
\vskip-0.15in
\begin{equation*}
\Delta(ab)=a\cdot\Delta(b)+\Delta(a)\cdot b\,\text{ for } a, b\in A.
\end{equation*}
If an infinitesimal bialgebra has an antipode $S$,  then it will be called an infinitesimal Hopf algebra~\mcite{MA}.
The basic theory of infinitesimal bialgebras and infinitesimal Hopf algebras was developed by Aguiar~\mcite{MA, Agu01, Agu02, Aguu02}, which has proven useful not only in combinatorices~\mcite{Agu02}, but in other areas of mathematics as well, such as associative Yang-Baxter equations, Drinfeld's doubles and pre-Lie algebras~\mcite{MA}.
Recently,  Wang~\mcite{WW14} generalized  Aguiar's result by developing the Drinfeld's double for braided infinitesimal Hopf algebras in Yetter-Drinfeld categories.

We emphasize that another version of infinitesimal bialgebras and infinitesimal Hopf algebras was defined by Loday and Ronco~\mcite{LR06}
and brought new life on rooted trees by Foissy~\mcite{Foi09, Foi10} in the sense that
$$\Delta(ab)=a\cdot\Delta(b)+\Delta(a)\cdot b-a\ot b\,\text{ for } a, b\in A.$$
In~\mcite{GZ}, the authors combined the two versions of infinitesimal bialgebras by defining the coproduct of $A$ to be the following compatibility:
$$\Delta(ab)=a\cdot\Delta(b)+\Delta(a)\cdot b+\lambda (a\ot b)\,\text{ for } a, b\in A,$$
where $\lambda\in \bfk$ is a fixed constant. This leads to the born of weighted infinitesimal (unitary) bialgebras, that is, the infinitesimal (unitary) bialgebras of weight $\lambda$. See Definition~\mref{def:iub} below.

In the present paper, we equip a matrix algebra with a weighted infinitesimal unitary bialgebraic structure by a construction of a suitable coproduct. Moreover, we also equip an infinitesimal unitary bialgebra of weight zero on matrix algebras with an antipode such that it is further an infinitesimal unitary Hopf algebra, under the view of Aguiar~\mcite{MA}.

Pre-Lie algebras, also called Vinberg algebras, first appeared in the work of Vinberg~\mcite{Vin63} under the name left-symmetric algebras on convex homogeneous cones and also appeared independently at the same time in the study of affine structures on
manifolds~\mcite{Ger63}. Its study has a broad applications in mathematices and mathematical physics, such as classical and quantum Yang-Baxter equations~\mcite{Bai07, Bai08, Bor90, ES99, GS00}, pre-Poisson algebras~\mcite{Agu00} and Poisson brackets~\mcite{BN85}, quantum field theory~\mcite{CK98, CK00, GPZ1, Kre98, Kre03, TK13} and operads~\mcite{CL01, PBG17}, Lie group and Lie algebras~\mcite{ Kim86, Med81, Vin63}, $\mathcal{O}$-operators~\mcite{Bai, Bai07, BGN12} and Rota-Baxter algebras~\mcite{AB08, BBGN13, EGP07, EGK05, GG, GGR}. In~\mcite{Bai}, Bai pointed out that ``Due to the nonassociativity of pre-Lie algebra, there is not a suitable (and computable) representation theory
and not a complete (and good) structure theory of pre-Lie algebras". It is nature to consider how to construct them from some algebraic structures (especially associative algebras) which we have known. Our results may give a new and elementary method to construct pre-Lie algebraic structures on some associative algebras, especially on matrix algebras.

It should be pointed out that Bai~\mcite{Bai} gave two approaches to construct pre-Lie algebras from associative algebras, see~\cite[Corollary~3.15]{Bai} and ~\cite[Proposition~3.16]{Bai} for more details.
Our consideration of pre-Lie algebras on matrix algebras has motivations beyond a simple pursuit of weighted infinitesimal unitary bialgebras on matrix algebras. In the algebraic framework of Aguiar~\mcite{Aguu02} for infinitesimal bialgebra , a pre-Lie algebraic structure is constructed from an arbitrary infinitesimal bialgebra. Motivated by Aguiar's construction, we previously derive a pre-Lie algebra from an arbitrary weighted infinitesimal bialgebra. As applications, two pre-Lie algebras on matrix algebras are built in the present paper.

Due to the construction of  an infinitesimal unitary bialgebra of weight $\lambda$ arising from a matrix algebra in this paper,
there is a close relationship among the  matrix algebras, pre-Lie algebras, Lie algebras and weighted infinitesimal unitary bialgebras. This situation can be summarized in the sense of following commutative diagram of categories
\begin{align*}
\xymatrix@C0.8em{
\text{\small{Matrix algebras} }\ar[d]_{} \ar[r]^{} &\text{\small{Weighted infinitesimal (unitary) bialgebras} } \ar[d]^{} \\
\text{\small{Lie algebras}} &  \ar[l]^{}  \text{\small{Pre-Lie algebras} }}
\end{align*}

Let $\mathfrak{g}$ be a Lie algebra and $r\in \mathfrak{g} \ot \mathfrak{g}$. A well-known result about classical Yang-Baxter equation (CYBE) studied by Drinfeld~\mcite{Dri86} is that the principal derivation $\delta_r: \mathfrak{g} \rightarrow \mathfrak{g} \ot \mathfrak{g}$ is coassociative if and only if the element
\begin{align*}
[r_{12}, r_{13}]+[r_{12}, r_{23}]+[r_{13}, r_{23}]\in \mathfrak{g} \ot \mathfrak{g} \ot \mathfrak{g}
\end{align*}
is $\mathfrak{g}$-invariant. The solutions of classical Yang-Baxter equation give rise to Lie bialgebras and quantum groups~\mcite{Dri86}. Parallel to classical Yang-Baxter equation, Aguiar~\mcite{MA} introduced associative Yang-Baxter equation (AYBE)
\begin{align*}
r_{13}r_{12}-r_{12}r_{23}+r_{23}r_{13}=0.
\end{align*}
Aguiar~\mcite{Agu01} shown that any solution $r$ of AYBE in an algebra $A$ is a solution of CYBE in $A^{lie}$ provided that $r+\tau(r)$ is $A$-invariant. Here $A^{lie}$ denotes the Lie algebra obtained by endowing $A$ with the commutator bracket and $\tau$ is the switch map. In particular, any skew-symmetric solution of AYBE is a skew-symmetric solution of CYBE in $A^{lie}$.

Let $A$ be a unitary algebra. For each solution $r\in A\ot A$, the principle derivation
\begin{align*}
\Delta_{r}: A\rightarrow A\ot A, \quad a\mapsto a\cdot r-r\cdot a
\end{align*}
endows $A$ with an infinitesimal unitary bialgebra of weight zero.
In this paper, we generalize Aguiar's result by the following result.

\begin{theorem}\rm(={\bf Theorem}$\mathrm{~\ref{thm:iff}}$\rm)
Let $A$ be a unitary algebra and $r=\sum_i u_i \ot v_i\in A\ot A$. Then the weighted principle derivation
$\Delta_{r}(a)=a\cdot r-r\cdot a-\lambda (a\ot 1)$
is coassociative if and only if the element
$r_{13}r_{12}-r_{12}r_{23}+r_{23}r_{13}-\lambda r_{13}\in A\ot A\ot A$
is $A$-invariant.
\end{theorem}
\noindent Then we call the equation
\begin{align*}
r_{13}r_{12}-r_{12}r_{23}+r_{23}r_{13}=\lambda r_{13}
\end{align*}
an associative Yang-Baxter equation of weight $\lambda$. Let us emphasize that the weighted AYBEs considered here are very general, which include the AYBEs~\mcite{MA},
the modified AYBEs~\mcite{EF02} and the non-homogeneous AYBEs  in~\mcite{BGN12}. See Remark~\mref{rem:3ex} below.
 Moreover, a surprising phenomenon shows that a solution of a weighted AYBE induces a Rota-Baxter operator of weight $-\lambda$ (Theorem~\mref{thm:RB2}). As a consequence, we give a bijection between the set of the solutions of weighted AYBE and the set of Rota-Baxter operators on matrix algebras.
\begin{theorem}\rm(={\bf Theorem}$\mathrm{~\ref{thm:bijection}}$\rm)
Let $r$ be a solution of an  AYBE of weight $\lambda$ in $M_{n}(\bfk)$. Then the map $r\rightarrow P_r$ is a bijection between the set of the solutions of AYBE of weight $\lambda$ in $M_{n}(\bfk)$ and the set of Rota-Baxter operators of weight $-\lambda$ on $M_{n}(\bfk)$ .
\end{theorem}

The solutions of an AYBE of weight $\lambda$ give rise to an infinitesimal unitary bialgebra of weight $\lambda$ and we call this infinitesimal unitary bialgebra the weighted quasitriangular infinitesimal unitary bialgebra.
As an application, we can derive a dendriform algebra from a weighted quasitriangular infinitesimal unitary bialgebra and obtain a commutative diagram:
\begin{align*}
\xymatrix@C0.8em{
 \text{\small{Weighted quasitriangular $\epsilon$-unitary bialgebras} }\ar[d]_{} \ar[r]^{} & \text{\small{Weighted  $\epsilon$-unitary bialgebras}}  \ar[rd]_{}\\
\text{\small{Rota-Baxter algebras} } \ar[r]^{}&\text{\small{Dendriform algebras}}\ar[r]^{}& \text{\small{Pre-Lie algebras} }}
\end{align*}
{\bf Structure of the Paper.} In Section~\mref{sec:infbi}, we start by recalling the concept of an infinitesimal (unitary) bialgebra of weight $\lambda$ (Definition~\mref{def:iub}). Then we proceed to construct two different coproducts $\col$ and $\coll$ on matrix algebra $M_n(\bfk)$ to equip them with two infinitesimal unitary bialgebraic structrues (Theorems~\mref{thm:fma} and ~\mref{thm:fma1}).
At the end of this section, we equip an infinitesimal unitary Hopf algebra on matrix algebras (Theorem~\mref{thm:rt13}), under the view of Aguiar~\mcite{MA}.

In Section~\mref{sec:prelie},
by investigating the relationship between weighted infinitesimal bialgebras and pre-Lie algebras (Theorem~\mref{thm:preL}), we equip $M_{n}(\bfk)$ with two pre-Lie bialgebraic structures and two Lie algebraic structures (Theorems~\mref{thm:preope11} and~\mref{thm:preope12}), beyond a construction of coproducts. It should be pointed out that the new Lie bracket on $M_n({\bfk})$ induced by $\col$ is different from the classical one (Example~\mref{exam:dif}), and the Lie bracket derived from $\coll$ is precisely the classical Lie bracket on matrix algebra when $\lambda =1$ (Remark~\mref{rem:liesam}).

In Section~\mref{sec:wei}, we introduce the concept of a weighted associative Yang-Baxter equation (Definition~\mref{def:WAYBE}), which generalizes the concept of associative Yang-Baxter equation studied in~\cite[Section~5]{MA}. We propose the concept of a weighted principle derivation $\Delta_r$ and characterize $\Delta_r$ to be coassociative.
We show that if $r$ is a solution of a weighted AYBE in $A$, then the quadruple  $(A, m, 1, \Delta_r)$ is an $\epsilon$-unitary bialgebra of weight $\lambda$ (Theorem~\mref{thm:mainiff}). As an application,  a solution of a homogeneous associative Yang-Baxter equation for matrix algebras is also given (Theorem~\mref{thm:solu}). We also derive a Rota-Baxter operator of weight $-\lambda$ from a solution of an  AYBE of weight $\lambda$ (Theorem~\mref{thm:RB2}).
We end this section by giving a  one-to-one correspondence between the solutions of the associative Yang-Baxter equation of weight $\lambda$  and Rota-Baxter operators of weight $-\lambda$ on matrix algebras (Theorem~\mref{thm:bijection}).

In Section~\mref{sec:qua}, we first propose the concept of weighted quasitriangular $\epsilon$-unitary bialgebras (Definition~\mref{def:qua}). Similar to the classical quasitriangular bialgebras, we then give some properties of weighted quasitriangular $\epsilon$-unitary bialgebras (Proposition~\mref{prop:quaiff}). We finally derive a dendriform algebra from a weighted quasitriangular $\epsilon$-unitary bialgebra (Theorem~\mref{thm:quadri}).

{\bf Notation.}
Throughout this paper, let $\bfk$ be a unitary commutative ring unless the contrary is specified,
which will be the base ring of all modules, algebras, coalgebras, bialgebras, tensor products, as well as linear maps.
By an algebra we mean an associative \bfk-algebra (possibly without unit)
and by a coalgebra we mean a coassociative \bfk-coalgebra (possibly without counit).
For an algebra $A$, we view $A\ot A$ as an $A$-bimodule via
\begin{equation}
a\cdot(b\otimes c):=ab\otimes c\,\text{ and }\, (b\otimes c)\cdot a:= b\otimes ca,
\mlabel{eq:dota}
\end{equation}
where $a,b,c\in A$.

\section{Weighted infinitesimal unitary bialgebras and examples}\label{sec:infbi}
In this section, we first recall the concept of a weighted infinitesimal (unitary) bialgebra~\mcite{GZ},
which generalise simultaneously the one introduced by Joni and Rota~\mcite{JR} and the one initiated by Loday and Ronco~\mcite{LR06}.
Then we proceed to equip a matrix algebra with a weighted infinitesimal unitary bialgebraic structure, in terms of a construction of a suitable coproduct.

\subsection{Weighted infinitesimal unitary bialgebras}
The following is the concept of a weighted infinitesimal (unitary) bialgebra proposed in~\mcite{GZ}.

\begin{defn}\mcite{GZ}
Let $\lambda$ be a given element of $\bfk$.
An {\bf infinitesimal bialgebra} (abbreviated {\bf $\epsilon$-bialgebra}) {\bf of weight $\lambda$} is a triple $(A,m,\Delta)$
consisting of an algebra $(A,m)$ (possibly without unit) and a coalgebra $(A,\Delta)$ (possibly without counit) that satisfies
\begin{equation}
\Delta (ab)=a\cdot \Delta(b)+\Delta(a) \cdot b+\lambda (a\ot b)\, \text{ for }\, a, b\in A.
\mlabel{eq:cocycle}
\end{equation}
If further $(A,m,1)$ is a unitary algebra, then the quadruple $(A,m,1, \Delta)$ is called an {\bf infinitesimal unitary bialgebra} (abbreviated {\bf $\epsilon$-unitary bialgebra}) {\bf of weight $\lambda$}.
\mlabel{def:iub}
\end{defn}

The concept of an $\epsilon$-bialgebra morphism is given as usual.

\begin{defn}\mcite{GZ}
Let $A$ and  $B$ be two $\epsilon$-bialgebras of weight $\lambda$.
A map $\phi : A\rightarrow B$ is called an {\bf infinitesimal bialgebra morphism} (abbreviated $\epsilon$-bialgebra morphism) if $\phi$ is an algebra morphism and a coalgebra morphism. The concept of an {\bf infinitesimal unitary bialgebra morphism} can be defined in the same way.
\end{defn}

\begin{remark}\mlabel{remk:unit1}
\begin{enumerate}
\item Let $(A,m,1, \Delta)$ be an $\epsilon$-unitary bialgebra of weight $\lambda$. Then $\Delta(1)=-\lambda(1\ot1)$ by
\begin{align*}
\Delta(1)=\Delta(1\cdot1)=1 \cdot \Delta(1) +\Delta(1)\cdot 1 +\lambda (1\ot 1)=2\Delta(1)+\lambda (1\ot 1).
\end{align*}

\item \mlabel{remk:b}
Aguiar~\mcite{MA} pointed out that there is no non-zero $\epsilon$-bialgebra of weight zero which is both unitary and counitary.
Indeed, it follows from the counicity that
$$1\ot 1_{\bfk}=(\id \ot \epsilon)\Delta(1)=0,$$ and so $1=0$.
\end{enumerate}

\end{remark}

\begin{exam}\label{exam:bialgebras}
Here are some examples of $\epsilon$-(unitary) bialgebras.
\begin{enumerate}
\item Any unitary algebra $(A, m,1)$ is an $\epsilon$-unitary bialgebra of weight zero by taking $\Delta=0$.
\item \cite[Example~2.3.5]{MA} The {\bf polynomial algebra} $\bfk \langle x_1, x_2, x_3,\ldots \rangle$ is an $\epsilon$-unitary bialgebra of weight zero with the coproduct $\Delta$
given by Eq.~(\mref{eq:cocycle}) and
 \begin{align*}
    \Delta (x_n)=\sum_{i=0}^{n-1}x_{i}\ot x_{n-1-i}=1\ot x_{n-1}+x_1\ot x_{n-2}+ \cdots +x_{n-1}\ot 1,
    \end{align*}
where we set $x_0=1$.
\item \cite[Example~2.3.2]{MA}Let $Q$ be a quiver. The {\bf path algebra} of $Q$ is the associative algebra $\bfk Q =\oplus_{n=0}^{\infty}\bfk Q_n$ whose underlying $\bfk $-module has its basis the set of all paths $a_1a_2\cdots a_n$ of length $n\geq 0$ in $Q$. The multiplication $\ast$ of two paths $a_1a_2\cdots a_n$ and $b_1b_2\cdots b_m$ is defined by
\begin{align*}
(a_1a_2\cdots a_n)\ast(b_1b_2\cdots b_m):=\delta_{t(a_n),s(b_1)}a_1a_2\cdots a_nb_1b_2\cdots b_m,
\end{align*}
where $\delta_{t(a_n),s(b_1)}$ is the Kronecker delta. The path algebra $(\mathbf{k}Q, \ast, \Delta)$ is an $\epsilon$-bialgebra of weight zero with the coproduct defined by
\begin{align*}
\Delta(e):&=0 \, \text { for } e\in Q_0 \\
\Delta(a):&=s(a)\ot t(a)\,\,\, \text { for } a\in Q_1, \text{and}\\
\Delta(a_1a_2\cdots a_n):=s(a_{1})\ot a_2 \cdot\cdots a_n&+a_1\cdots a_{n-1}\ot t(a_n)+\sum_{i=1}^{n-2}a_1\cdots a_i\ot a_{i+2}\cdots a_n \text { for } n\geq 2.
\end{align*}

\item \cite[Section~2.3]{LR06}\label{exam:tensor}
Let $V$ denote a vector space. Recall that the {\bf tensor algebra} $T(V)$ over $V$ is the tensor module,
\begin{align*}
T(V)=\bfk \oplus V\oplus V^{\ot 2}\oplus \cdots \oplus V^{\ot n}\oplus \cdots,
\end{align*}
equipped with the associative multiplication $m_{T}$ called concatenation defined by
\begin{align*}
v_1\cdots v_i\ot v_{i+1}\cdots v_n \mapsto v_1\cdots v_i v_{i+1}\cdots v_n \quad \text{ for } 0 \leq i\leq n,
\end{align*}
 with the convention that $v_1v_0=1$ and $v_{n+1}v_{n}=1$. It is a well-known free associative algebra. The
tensor algebra $T(V)$ is an $\epsilon$-unitary bialgebra of weight $-1$ with the coproduct defined by
\begin{align*}
\Delta(v_1\cdots v_n):=\sum_{i=0}^{n}v_1\cdots v_i\ot v_{i+1}\cdots v_n.
\end{align*}

\item \cite[Propsition~2.6]{GZ}
The {\bf polynomial algebra} $\bfk [x]$ is an $\epsilon$-unitary bialgebra of weight $\lambda$ with the coproduct defined by
    \begin{align*}
    \Delta(1):=-\lambda (1\ot 1) \text{ and }
    \Delta(x^n):=\sum_{i=0}^{n-1}x^{i}\ot x^{n-1-i}+\lambda \sum_{i=1}^{n-1}x^{i}\ot x^{n-i} \quad \text { for } n\geq 1.
    \end{align*}
 \item \cite[Section~1.4]{Foi08} Let $(A, m, 1,\Delta, \varepsilon, c)$ be a braided bialgebra with $A = \bfk\oplus \ker\varepsilon$ and the braiding
$c:A\ot A \to A\ot A$ given by
\begin{align*}
c:\left\{
    \begin{array}{ll}
    1\ot 1 \mapsto 1\ot 1,\\
   a\ot 1\mapsto 1\ot a, \\
    1\ot b \mapsto b\ot 1,\\
    a\ot b \mapsto 0,
    \end{array}
    \right.
\end{align*}
where $a,b\in \ker \varepsilon$.
Then $(A, m, 1,\Delta, \varepsilon)$ is an $\epsilon$-unitary bialgebra of weight $-1$.

\end{enumerate}
\end{exam}

\subsection{An infinitesimal unitary bialgebra on a matrix algebra}\label{sec:sub}
In this subsection, we construct an $\epsilon$-unitary bialgebra of weight $\lambda$ arising from a matrix algebra.

\begin{defn}\cite[Chapter~17]{Lam99}
 A {\bf matrix algebra} $M_n(\bfk)$ is a collection of $n\times n$ matrices over $\bfk$ that form a unitary associative algebra under matrix addition and matrix multiplication.
\end{defn}

The multiplication on $M_n(\bfk)$ will be denoted by $\frakm$.
We now define a coproduct on  matrix algebra $M_n(\bfk)$  to equip it with a coalgebra structrue, with an eye toward constructing an $\epsilon$-unitary bialgebra of weight $\lambda$ on it.
\subsubsection{The case of $\lambda=0$}
Let $L\in M_n(\bfk)$ such that $L^2=0$.
For $M\in M_n(\bfk)$, define
\begin{align}
\col (M):=ML\ot L- L\ot LM.
\mlabel{eq:col}
\end{align}
Note that $\col(E)=0$, where $E\in M_n(\bfk)$ is the identity matrix. We observe that $M_n(\bfk)$ is closed under the coproduct $\col$.

\begin{lemma}\label{lem:comp1}
Let $L\in M_n(\bfk)$ such that $L^2=0$. Then for $M, N \in M_n(\bfk)$,
\begin{align*}
\col(MN)=M\cdot\col(N)+\col(M)\cdot N.
\end{align*}

\end{lemma}
\begin{proof}
We have
\begin{align*}
M\cdot\col(N)+\col(M)\cdot N=&M\cdot(NL\ot L-L\ot LN)+(ML\ot L-L\ot LM)\cdot N   \quad (\text{by Eq.~(\mref{eq:col})})\\
=&MNL\ot L-ML\ot LN+ML\ot LN-L\ot LM N \quad (\text{by Eq.~(\ref{eq:dota})})\\
=&MNL\ot L-L\ot LM N\\
=&(MN)L\ot L-L\ot L(M N)\\
=&\col(MN) \quad (\text{by Eq.~(\ref{eq:col})}).
\end{align*}
This completes the proof.
\end{proof}

\begin{lemma}
Let $L\in M_n(\bfk)$ such that $L^2=0$.
Then the pair $(M_n(\bfk), \col)$ is a coalgebra (without counit). \mlabel{lem:coal}
\end{lemma}

\begin{proof}
It is enough to show the coassociative law:
\begin{align}
(\id\otimes \col)\col(M)=(\col\otimes \id)\col(M)\,\text{ for } M\in M_n(\bfk).
\mlabel{eq:coassp}
\end{align}
Applying Eq.~(\mref{eq:col}) and $L^2=0$, we have
\begin{align}
\col(L)=L^2\ot L-L\ot L^2=0.
\mlabel{eq:coll}
\end{align}
On the one hand,
\begin{align*}
(\id\otimes \col)\col(M)=&\ (\id\otimes \col)(ML\ot L- L\ot LM) \quad (\text{by Eq.~(\ref{eq:col})})\\
=&\ ML\ot \col(L)- L\ot \col(LM)\\
=&\  - L\ot \col(LM) \quad (\text{by Eq.~(\ref{eq:coll})})\\
=&\ - L\ot (LML\ot L- L\ot L^2M)\quad (\text{by Eq.~(\ref{eq:col})})\\
=&\ - L\ot LML\ot L \quad (\text{by $L^2=0$}).
\end{align*}
On the other hand,
\begin{align*}
(\col\otimes \id)\col(M)=&\ (\col\otimes \id)(ML\ot L- L\ot LM) \quad (\text{by Eq.~(\ref{eq:col})})\\
=&\ \col(ML)\ot L- \col(L)\ot LM\\
=&\ \col(ML)\ot L \quad (\text{by Eq.~(\ref{eq:coll})})\\
=&\ (ML^2\ot L- L\ot LML)\ot L\quad (\text{by Eq.~(\ref{eq:col})})\\
=&\ - L\ot LML\ot L\quad (\text{by $L^2=0$}),
\end{align*}
whence Eq.~(\mref{eq:coassp}) holds.
\end{proof}

Now we arrive at our first main result in this subsection.

\begin{theorem}
Let $L\in M_n(\bfk)$ such that $L^2=0$. Then the quadruple $(M_n(\bfk), \frakm, E,\col)$ is an $\epsilon$-unitary bialgebra of weight zero.
\mlabel{thm:fma}
\end{theorem}
\begin{proof}
It follows from Lemmas~\mref{lem:comp1} and ~\mref{lem:coal}.
\end{proof}

\begin{remark}
\begin{enumerate}
\item It should be pointed out that Theorem~\mref{thm:fma} generalizes the case of $n=2$ studied in~\cite[Example~2.3.7]{MA}.

\item By Remark~\mref{remk:unit1}~(\mref{remk:b}), it cann't add a counit to $(M_n(\bfk), \frakm, E,\col)$.
\end{enumerate}
\end{remark}

\subsubsection{The case of $\lambda\neq 0$}
Let $\lambda\in \bfk\setminus\{0\}$ be given. For any $M\in M_n(\bfk)$, define
\begin{align}
\coll (M):=M\ot  (-\lambda E),
\mlabel{eq:col1}
\end{align}
where $E\in M_n(\bfk)$ is the identity matrix.
Note that $\coll(E)=-\lambda (E\ot E)$ and $M_n(\bfk)$ is closed under the coproduct $\coll$.

\begin{lemma}\label{lem:comp2}
Let $M, N \in M_n(\bfk)$. Then
\begin{align*}
\coll(MN)=M\cdot\coll(N)+\coll(M)\cdot N+\lambda (M\ot N).
\end{align*}
\end{lemma}

\begin{proof}
We have
\begin{align*}
M\cdot\coll(N)+\coll(M)\cdot N&=M\cdot \Big(N\ot  (-\lambda E) \Big)+\Big(M\ot  (-\lambda)E\Big)\cdot N  \quad (\text{by Eq.~(\mref{eq:col1})}),\\
&=MN\ot (-\lambda E) -\lambda (M\ot N)\quad (\text{by Eq.~(\ref{eq:dota})})\\
&=\coll(MN)-\lambda( M\ot N).
\end{align*}
Thus
\begin{align*}
\coll(MN)=M\cdot\coll(N)+\coll(M)\cdot N+\lambda (M\ot N).
\end{align*}
This completes the proof.
\end{proof}

\begin{lemma}
The pair $(M_n(\bfk), \coll)$ is a coalgebra (without counit). \mlabel{lem:coal1}
\end{lemma}

\begin{proof}
It suffices to show the coassociative law:
\begin{align}
(\id\otimes \coll)\coll(M)=(\coll\otimes \id)\coll(M)\,\text{ for } M\in M_n(\bfk).
\mlabel{eq:coassp1}
\end{align}
By Eq.~(\mref{eq:col1}), we have
\begin{align*}
(\id\otimes \coll)\coll(M)&=(\id\otimes \coll)\Big(M\ot  (-\lambda E)\Big)
=M\ot (-\lambda) \coll(E)=\lambda^2 M\ot E\ot E\\
&=M\ot (-\lambda E)\ot (-\lambda E)
=\coll(M)\ot (-\lambda E)\\
&=(\coll\ot \id)(M\ot (-\lambda E))=(\coll\otimes \id)\coll(M),
\end{align*}
as desired.
\end{proof}

\begin{theorem}
Let $\lambda\in \bfk\setminus \{0\}$ be given. Then the quadruple $(M_n(\bfk), \frakm, E,\coll)$ is an $\epsilon$-unitary bialgebra of weight $\lambda$.
\mlabel{thm:fma1}
\end{theorem}
\begin{proof}
It follows from Lemmas~\mref{lem:comp2} and ~\mref{lem:coal1}.
\end{proof}

\subsection{An infinitesimal unitary Hopf algebra on a matrix algebra}
In this subsection, we equip the $\epsilon$-unitary bialgebra  $(M_n(\bfk), m, E,\col)$ of weight zero with an antipode such that it is further an $\epsilon$-unitary Hopf algebra, under the view of Aguiar~\mcite{MA}.
Denote by $\Hom_{\bfk}(A,B)$ the module consisting of linear maps from $A$ to $B$ throughout the remainder of this section.

\begin{defn}~\cite[Chapter~4.1]{CP94}
Let $A=(A, m, 1,\Delta, \varepsilon)$ be a classical bialgebra. Then the convolution product $\ast $ on $\Hom_{\bfk}(A,A)$ is defined to be the composition:
\begin{align*}
f\ast g :=m (f\ot g)  \Delta \, \text { for }\,  f,g\in \Hom_{\bfk}(A,A),
\end{align*}
and the triple $(\Hom_{\bfk}(A,A), \ast, 1\circ \varepsilon)$ is called a convolution algebra, where $1\circ \varepsilon$ is the unit with respect to $\ast$.
The antipode $S$ is defined to be the inverse of the identity map with respect to the convolution product.
\end{defn}

\begin{remark}
The question facing us is whether we can define the antipode for an $\epsilon$-unitary bialgebra as one does for classical bialgebras $A$. Aguiar~\cite[Remark~2.2]{MA} answers this question `No' due to the lack of the unit $1\circ \varepsilon$ with respect to $\ast$, see Remark~\mref{remk:unit1}~(\mref{remk:b}).
\end{remark}
However, Aguiar~\mcite{MA} provided the perfect notion of antipode $S$ for an $\epsilon$-bialgebra.

\begin{defn}~\cite[Section~3]{MA}
Let $A=(A, m, \Delta)$ be an $\epsilon$-bialgebra. Then the {\bf circular convolution} $\circledast$ on $\Hom_{\bfk}(A,A)$ is defined by
$$f\circledast g :=f\ast g+f+g \, \text {, that is, } \, (f\circledast g)(a) :=\sum_{(a)}f(a_{(1)})g({a_{(2)}})+f(a)+g(a)\, \text{ for } a\in A.$$
Note that $f\circledast 0 = f = 0\circledast f$ and so $0\in \Hom_{\bfk}(A,A)$ is the unit with respect to the circular convolution $\circledast$.
\end{defn}

Further Aguiar introduced the concept of an infinitesimal Hopf algebra via circular convolution~\cite[Definition~3.1]{MA}.
Inspired by this, we set the following definition in~\mcite{GZ}.

\begin{defn}
 An infinitesimal unitary bialgebra $(A, m, 1, \Delta)$ of weight $\lambda$ is called an {\bf infinitesimal unitary Hopf algebra} (abbreviated {\bf $\epsilon$-unitary Hopf algebra}) of {\bf weight $\lambda$} if the identity map $\id\in \Hom_{\bfk}(A,A)$ is invertible with respect to the circular convolution. In this case, the inverse $S\in \Hom_{\bfk}(A,A)$ of $\id$ is called the {\bf antipode} of $A$. It is characterized by
 the equations
\begin{equation}
\sum_{(a)}S(a_{(1)})a_{(2)}+S(a)+a=0=\sum_{(a)}a_{(1)}S(a_{(2)})+a+S(a) \, \text{ for }\, a \in A,
\mlabel{eq:ants}
\end{equation}
where $\Delta(a) = \sum_{(a)} a_{(1)} \ot a_{(2)}$.
\mlabel{de:deha}
\end{defn}

The $\epsilon$-unitary Hopf algebra satisfies many properties analogous to those of a classical Hopf algebra~\cite[Propositions~3.7, 3.12]{MA}.

\begin{remark}
\begin{enumerate}
\item Let $A$ be an $\epsilon$-unitary Hopf algebra of weight zero with antipode $S$. Then
\begin{align*}
S(xy)=-S(x)S(y) \text{ and }\sum_{(x)} S(x_{(1)})\ot S(x_{(2)})=-\sum_{(S(x))}S(x)_{(1)} \ot S(x)_{(2)} =-\Delta S(x).
\end{align*}
\item If $(A,m,1, \Delta)$ is an $\epsilon$-unitary Hopf algebra of weight zero with the antipode $S$, then so is $(A,m^{\mathrm{op}}, 1,\Delta^{\mathrm{cop}})$ with the same antipode $S$.

\item It follows from  Eq.~(\mref{eq:ants}) that $S(1)=-1$ by taking $a=1$.
\end{enumerate}
\end{remark}

\begin{defn}\cite[Section~4]{MA}
Let $A$ be an algebra and $C$ a coalgebra. The map $f:C\to A$ is called {\bf locally nilpotent} with respect to convolution $*$
if for each $c\in C$ there is some $n\geq 1$ such that
\begin{equation}\mlabel{eq:lnil}
f^{\ast n}(c) :=\sum_{(c)}f(c_{(1)})f(c_{(2)})\cdots f(c_{(n+1)})=0,
\end{equation}
where $c_{(1)}, \cdots, c_{(n+1)}$ are from the Sweedler notation $\Delta^n(c) = \sum_{(c)} c_{(1)} \ot \cdots \ot c_{(n+1)}$
and $f^{\ast n}$ is defined inductively by
\begin{align*}
f^{\ast 1}(c) :=\sum_{(c)}f(c_{(1)})f(c_{(2)})\, \text{ and }\,
f^{\ast (n)}:=f^{\ast (n-1)}\ast f.
\end{align*}
\end{defn}

\begin{lemma}
Let $(M_n(\bfk), \, \frakm,\,E,\col)$ be the $\epsilon$-unitary bialgebra of weight zero in Theorem~\mref{thm:fma} and
$$D:= \frakm \col: M_n(\bfk) \to M_n(\bfk).$$
Then for each $k\geq 0$ and $M\in M_n(\bfk)$,  $D^{\ast (k+1)}(M)=0$ and so $D$ is locally nilpotent.
\mlabel{lem:rt12}
\end{lemma}

\begin{proof}
It suffices to prove the first statement by induction on $k\geq 0$.
Using sweedler notation, we may write
\begin{align*}
 \col(M)=\sum_{(M)}M_{(1)}\ot M_{(2)} \, \text { for }\, M\in M_n(\bfk).
\end{align*}
 For the initial step of $k=0$, we have
\begin{align*}
D^{*(1)}(M)=\sum_{(M)}D(M_{(1)})D(M_{(2)})=0,
\end{align*}
where the last step follows from
\begin{align*}
D(M_{(1)})= \frakm \col (M_{(1)})=M_{(1)}L^2-L^2M_{(1)}=0
\end{align*}
by Eq.~(\mref{eq:col}) and $L^2=0$.
Assume the result is true for $k=\ell$ for an $\ell\geq 1$, and consider the case when $k=\ell+1.$ Then
\begin{align*}
D^{\ast (\ell+1)}(M)=&\ (D^{\ast \ell}\ast D)(M)= \frakm (D^{\ast \ell}\ot D)\col(M)=\sum_{(M)}D^{\ast \ell}(M_{(1)})D(M_{(2)})=0,
\end{align*}
where the last step employs the induction hypothesis.
This completes the proof.
\end{proof}

Denote by $\mathbb{R}$ and $\mathbb{C}$ the field of real numbers and the field of complex numbers, respectively.

\begin{lemma}\cite[Proposition~4.5]{MA}
Let $(A,\,m,\,\Delta)$ be an $\epsilon$-bialgebra and $D :=m\Delta$, and let $\bfk$ be a field. Suppose that either
\begin{enumerate}
\item
$\bfk = \mathbb{R}$ or $\mathbb{C}$ and $A$ is finite dimensional, or
\item
$D$ is locally nilpotent and char$(\bfk)=0$.
\end{enumerate}
Then $A$ is an $\epsilon$-Hopf algebra with bijective antipode $S=-\sum_{n=0}^{\infty}\frac{1}{n!}(-D)^{n}$.
\mlabel{lem:rt3}
\end{lemma}
Indeed, in the above lemma, we can replace the condition that $\bfk$ is a field by $\QQ\subseteq \bfk$.
\begin{theorem}
Let $\QQ\subseteq \bfk$. Then the quadruple $(M_n(\bfk), \, \frakm,\,E, \, \col)$  is an $\epsilon$-unitary Hopf algebra of weight zero
with the bijective antipode $S=-\sum_{n=0}^{\infty}\frac{1}{n!}(-D)^{n}$.
\mlabel{thm:rt13}
\end{theorem}

\begin{proof}
By Theorem~\mref{thm:fma}, $(M_n(\bfk), \,\frakm,\,E,\, \col)$ is an $\epsilon$-unitary bialgebra of weight zero.
From Lemmas~\mref{lem:rt12},~\mref{lem:rt3} and Definition~\mref{de:deha}, $(M_n(\bfk), \,\frakm,\, E, \,\col)$ is an $\epsilon$-unitary Hopf algebra of weight zero
with bijective antipode $S=-\sum_{n=0}^{\infty}\frac{1}{n!}(-D)^{n}$.
\end{proof}

\section{Pre-Lie algebras on matrix algebras}\label{sec:prelie}
This section is devoted to recall that an $\epsilon$-unitary bialgebra of an arbitrary weight $\lambda$ gives rise to a pre-Lie algebra~\mcite{GZ}, which generalizes the construction of the pre-Lie algebra from an infinitesimal bialgebra~\mcite{Aguu02}. As a consequence, we equip $M_{n}(\bfk)$ with two pre-Lie bialgebraic structures and two Lie algebraic structures

\subsection{Pre-Lie algebras and weighted infinitesimal unitary bialgebras }
In this subsection, we first recall the concept of  pre-Lie algebras and then show the connection from weighted $\epsilon$-unitary bialgebras to pre-Lie algebras.

\begin{defn}~\cite{Man11}
A {\bf (left) pre-Lie algebra} is a $\bfk $-module $A$ together with a binary linear operation $\rhd: A\ot A \rightarrow A$ satisfying
\begin{align}
(a\rhd b)\rhd c-a\rhd (b\rhd c)=(b\rhd a)\rhd c-b\rhd (a\rhd c)\, \text{ for }\, a,b,c\in A.
\mlabel{eq:lpre}
\end{align}
\end{defn}

\begin{exam}
Here are two well-known pre-Lie algebras on dendriform dialgebras and Rota-Baxter algebras, respectively.
\begin{enumerate}
\item Let $(A, \prec, \succ )$ be a dendriform dialgebra. Then the multiplication $\star$ defined by $a\star b=a\prec b+a\succ b$ gives an associative algebra~\mcite{Aguu02}. In addition, define
    \begin{align*}
\rhd: A\ot A \to A, \, a\ot b \mapsto  a\succ b-a\prec b\, \text { for }\, a, b \in A.
\end{align*}
Then $A$ together with $\rhd$ is a pre-Lie algebra~\mcite{Aguu02}.

\item Let $(A, P)$ be a Rota-Baxter algebra of weight $\lambda$. If the weight $\lambda=0$, then the binary operation
\begin{align*}
    \rhd: A\ot A \to A, \, a\ot b \mapsto P(a)\cdot b-b\cdot P(a)\, \text{ for }\, a, b \in A,
\end{align*}
defines a pre-Lie algebra. If the weight $\lambda=-1$, then the binary operation
\begin{align*}
    \rhd: A\ot A \to A, \, a\ot b \mapsto P(a)\cdot b-b\cdot P(a)-x\cdot y \, \text{ for } \, a, b \in A,
\end{align*}
defines a pre-Lie algebra~\mcite{AB08}.
\end{enumerate}
\end{exam}

Let $(A, \rhd)$ be a pre-Lie algebra. For any $a\in A$, let
$$L_a:A\rightarrow A, \quad b\mapsto  a \rhd b$$
be the left multiplication operator. Let
$$L:A\rightarrow \Hom_{\bfk}(A, A),\quad a\mapsto L_a.$$
The close relation between pre-Lie algebras and Lie algebras is characterized by the following two fundamental properties.

\begin{lemma}
\begin{enumerate}
\item
\cite[Theorem~1]{Ger63}
Let $(A, \rhd)$ be a pre-Lie algebra. Define for elements in $A$ a new multiplication by setting
\begin{align*}
[a, b] :=a\rhd b-b\rhd a\,\text{ for }\, a,b\in A.
\end{align*}
Then $(A, [_{-}, _{-}])$ is a Lie algebra.
\mlabel{lem:preL1}
\item \cite[Proposition~1.2]{Bai}
Eq.~(\mref{eq:lpre}) rewrites as
\begin{align*}
L_{[a,b]}=L_a\circ L_b -  L_b\circ L_a=[L_a, L_b],
\end{align*}
which implies that $L: (A, [_{-}, _{-}]) \rightarrow \Hom_{\bfk}(A, A)$ with $a\mapsto L_a$ gives a representation of the Lie algebra $(A, [_{-}, _{-}])$.
\mlabel{lem:preL2}
\end{enumerate}
\mlabel{lem:preL}
\end{lemma}

\begin{remark}
By Lemma~\mref{lem:preL}, a pre-Lie algebra induces a Lie algebra whose left multiplication operators give a representation of the associated commutator Lie algebra.
\end{remark}

Let $(A, m, 1, \Delta)$ be an $\epsilon$-unitary bialgebra of weight $\lambda$. Define
\begin{align}
\rhd: A\ot A \to A, \, a\ot b \mapsto a\rhd b:=\sum_{(b)}b_{(1)}a b_{(2)},
\mlabel{eq:preope}
\end{align}
where $b_{(1)}$ and $b_{(2)}$ are from the Sweedler notation $\Delta (b)=\sum_{(b)}b_{(1)}\ot b_{(2)}$.
The following result captures the connection from weighted $\epsilon$-unitary bialgebras to pre-Lie algebras~\mcite{GZ}.
For the completeness, we record the proof here.
\begin{theorem}
Let $(A, m, 1, \Delta)$ be an $\epsilon$-unitary bialgebra of weight $\lambda$.
Then $A$ equipped with the $\rhd$ in Eq.~(\mref{eq:preope}) is a pre-Lie algebra.
\mlabel{thm:preL}
\end{theorem}

\begin{proof}
For any $a,b, c\in A$, using the Sweedler notation $\Delta(c)=\sum_{(c)}c_{(1)}\ot c_{(2)}$, we have
\begin{align*}
\Delta(\sum_{(c)}c_{(1)}bc_{(2)})=&\sum_{(c)}\Delta(c_{(1)}bc_{(2)})=\sum_{(c)}\bigg(c_{(1)}b\cdot \Delta(c_{(2)})+\Delta(c_{(1)}b)\cdot c_{(2)}+\lambda c_{(1)}b \ot c_{(2)}\bigg)\ \ (\text{by Eq.~(\ref{eq:cocycle})})\\
=&\sum_{(c)}\bigg(c_{(1)}b\cdot \Delta(c_{(2)})+\Big(c_{(1)}\cdot \Delta(b)+\Delta(c_{(1)})\cdot b+\lambda(c_{(1)}\ot b)\Big)\cdot c_{(2)}+\lambda c_{(1)}b \ot c_{(2)}\bigg)\\
=&\sum_{(c)}\bigg(c_{(1)}b\cdot \Delta(c_{(2)})+c_{(1)}\cdot \Delta(b)\cdot c_{(2)}+\Delta(c_{(1)})\cdot b c_{(2)}+\lambda c_{(1)}\ot b c_{(2)}+\lambda c_{(1)}b \ot c_{(2)}\bigg)\\
=&\sum_{(c)}\bigg(\Delta(c_{(1)})\cdot b c_{(2)}+c_{(1)}\cdot \Delta(b)\cdot c_{(2)}+c_{(1)}b\cdot \Delta(c_{(2)})+\lambda c_{(1)}\ot b c_{(2)}+\lambda c_{(1)}b \ot c_{(2)}\bigg)\\
=&\sum_{(c)}\bigg(\sum_{(c_{(1)})}c_{(1)(1)}\ot c_{(1)(2)}bc_{(2)}+\sum_{(b)}c_{(1)}b_{(1)}\ot b_{(2)}c_{(2)}+\sum_{(c_{(2)})}c_{(1)}bc_{(2)(1)}\ot c_{(2)(2)}\\
&+\lambda c_{(1)}\ot bc_{(2)}+\lambda c_{(1)}b\ot c_{(2)}\bigg)\\
=&\sum_{(c)}c_{(1)}\ot c_{(2)}bc_{(3)}+\sum_{(c),(b)}c_{(1)}b_{(1)}\ot b_{(2)}c_{(2)}+\sum_{(c)}c_{(1)}bc_{(2)}\ot c_{(3)}\\
&+\lambda \sum_{(c)} c_{(1)}\ot bc_{(2)}+ \lambda \sum_{(c)}c_{(1)}b\ot c_{(2)}\quad \text{(by the coassociative law)}.
\end{align*}
Together this with Eq.~(\mref{eq:preope}), we obtain
\begin{align*}
a\rhd (b\rhd c)=& a\rhd (\sum_{(c)}c_{(1)}b c_{(2)})\\
=& \sum_{(c)}c_{(1)}a c_{(2)}bc_{(3)}+\sum_{(c),(b)}c_{(1)}b_{(1)}a b_{(2)}c_{(2)}+ \sum_{(c)}c_{(1)}bc_{(2)}a c_{(3)} \\
&+ \lambda \sum_{(c)}c_{(1)}a bc_{(2)}+\lambda \sum_{(c)}c_{(1)}ba c_{(2)}.
\end{align*}
Moreover
\begin{align*}
(a\rhd b)\rhd c=(\sum_{(b)}b_{(1)}a b_{(2)})\rhd c=\sum_{(c),(b)}c_{(1)}b_{(1)}a b_{(2)}c_{(2)}.
\end{align*}
Thus
\begin{align}\label{eq:preid1}
(a\rhd b)\rhd c-a\rhd (b\rhd c)=&-\sum_{(c)}c_{(1)}a c_{(2)}bc_{(3)} - \sum_{(c)}c_{(1)}bc_{(2)}a c_{(3)}
-\lambda\sum_{(c)} \bigg(c_{(1)}a bc_{(2)}+c_{(1)}ba c_{(2)}\bigg).
\end{align}
Observe that Eq.~(\mref{eq:preid1}) is symmetric in $a$ and $b$. Hence
\begin{align*}
(a\rhd b)\rhd c-a\rhd (b\rhd c)=(b\rhd a)\rhd c-b\rhd (a\rhd c),
\end{align*}
and so $(A, \rhd)$ is a pre-Lie algebra.
\end{proof}

\subsection{Two pre-Lie and a new Lie algebraic structures on a matrix algebra}
In this subsection, we proceed to give two pre-Lie algebraic structures on a matrix algebra. Consequently, a new Lie algebraic structure on a matrix algebra is induced by Lemma~\mref{lem:preL}~(\mref{lem:preL1}).

\subsubsection{The case of $\lambda=0$}
By Theorem~\mref{thm:fma}, $(M_n(\bfk), \frakm, E,\col)$ is an $\epsilon$-unitary bialgebra of weight zero.
Applying Eq.~(\mref{eq:preope}), we define
\begin{align}
\rhd_{L}: M_n(\bfk) \ot M_n(\bfk) \rightarrow M_n(\bfk),\quad M\rhd_{L} N: =\sum_{(N)}N_{(1)}M N_{(2)} ,
\mlabel{eq:preid}
\end{align}
where $N_{(1)}$ and $N_{(2)}$ are from  $\col(N)=\sum_{(N)}N_{(1)}\ot N_{(2)}$.

\begin{theorem}\mlabel{thm:preope11}
Let $L\in M_n(\bfk)$ such that $L^2=0$. Then the pair $(M_n(\bfk), \rhd_{L})$ is a pre-Lie algebra and so
$(M_n(\bfk), [_{-}, _{-}]_{L})$ is a Lie algebra, where
\begin{align*}
[M, N]_{L}:=NLML+LNLM-MLNL-LMLN.
\end{align*}
\end{theorem}

\begin{proof}
By Theorems~\mref{thm:fma} and~\mref{thm:preL}, $(M_n(\bfk), \rhd_{L})$ is a pre-Lie algebra. The remainder follows from
Lemma~\mref{lem:preL}~(\mref{lem:preL1}) and
\begin{align*}
[M, N]_{L}&=M \rhd_{L} N-N\rhd_{L} M=\sum_{(N)}N_{(1)}M N_{(2)}-\sum_{(M)}M_{(1)}N M_{(2)}\quad (\text{by Eq.~(\ref{eq:preid})})\\
&=NLML-LMLN-(MLNL-LNLM)\quad (\text{by Eq.~(\ref{eq:col})})\\
&=NLML+LNLM-MLNL-LMLN.
\end{align*}
This completes the proof.
\end{proof}


The following example exposes that the Lie bracket $[_{-}, _{-}]_{L}$ is different from the one
given by commutator.

\begin{exam}\label{exam:dif}
Consider the matrix algebra $M_2(\bfk)$.
Let
\[
L=\left(
\begin{array}{cc}
 0 & 1 \\
  0 & 0 \\
  \end{array}
  \right), M=\left(
  \begin{array}{cc}
  a & b \\
   c & d \\
   \end{array}
   \right), N=\left(
  \begin{array}{cc}
  e & f \\
   g & h \\
   \end{array}
   \right).
\] Then $L^2 = 0$
and by Theorem~\mref{thm:preope11},
\begin{align*}
[M, N]_{L}=\left(
             \begin{array}{cc}
               0 & ec+gd-ag-ch \\
               0 & 0 \\
             \end{array}
           \right),
\end{align*}
which is different from the classical Lie bracket
\begin{align*}
[M,N] :=MN-NM=\left(
              \begin{array}{cc}
                bg-fc & af+bh-eb-fd \\
                ce+dg-ga-hc & cf-gb \\
              \end{array}
            \right).
\end{align*}
\end{exam}

\subsubsection{The case of $\lambda \neq 0$}
By Theorem~\mref{thm:fma1}, $(M_n(\bfk), \frakm, E,\coll)$ is an $\epsilon$-unitary bialgebra of weight $\lambda$.
Using Eq.~(\mref{eq:preope}), we define
\begin{align*}
\rhd_{\lambda}: M_n(\bfk) \ot M_n(\bfk) \rightarrow M_n(\bfk), \quad M\rhd_{\lambda} N: =\sum_{(N)}N_{(1)}M N_{(2)} ,
\end{align*}
where $N_{(1)}$ and $N_{(1)}$ are from $\coll(N)=\sum_{(N)}N_{(1)}\ot N_{(2)}$.

\begin{theorem}\mlabel{thm:preope12}
Let $\lambda\in \bfk\setminus\{0\}$ be given. Then the pair $(M_n(\bfk), \rhd_{\lambda})$ is a pre-Lie algebra
and so $(M_n(\bfk), [_{-}, _{-}]_{\lambda})$ is a Lie algebra, where
\begin{align*}
[M, N]_{\lambda}:=\lambda (MN-NM).
\end{align*}
\end{theorem}
\begin{proof}
By Theorems~\mref{thm:fma1} and~\mref{thm:preL}, $(M_n(\bfk), \rhd_{\lambda})$ is a pre-Lie algebra.
The remainder follows from Lemma~\mref{lem:preL}~(\mref{lem:preL1}) and Eq.~(\mref{eq:col1}).
\end{proof}

\begin{remark}
Taking $\lambda=1$ in Theorem~\mref{thm:preope12}, a surprising phenomenon shows that the Lie bracket derived
from $\coll$ is precisely the classical Lie bracket given by commutator on matrix algebras.
\mlabel{rem:liesam}
\end{remark}

\section{Weighted AYBEs and weighted $\epsilon$-unitary bialgebras}\label{sec:wei}
In this section, we first introduce the concept of a weighted associative Yang-Baxter equation, which generalize the results studied in \cite[Section~5]{MA}. As an application, we give a solution of a (homogeneous) associative Yang-Baxter equation  for matrix algebras. We finally derive a Rota-Baxter operator of weight $-\lambda$  from a solution of a weighted associative Yang-Baxter equation.

\subsection{Weighted AYBEs}
Let us first propose the concept of weighted associative Yang-Baxter equations.
Let $A$ be a unitary algebra. For an element $r=\sum_{i}u_i\ot v_i \in A\ot A$, we write
\begin{align*}
r_{12}=\sum_{i}u_i\ot v_i \ot 1, \, r_{13}=\sum_{i}u_i\ot 1\ot v_i , \, r_{23}=\sum_{i}1\ot u_i \ot v_i.
\end{align*}

\begin{defn}\label{def:WAYBE}
Let $\lambda$ be a given element of $\bfk$ and $A$ a unitary algebra.
\begin{enumerate}
\item The equation
\begin{align}
r_{13}r_{12}-r_{12}r_{23}+r_{23}r_{13}=\lambda r_{13} \mlabel{eq:AYBE}
\end{align}
is called the {\bf  associative Yang-Baxter equation (AYBE for short) of weight $\lambda$}.
\item
An element $r=\sum_{i}u_i\ot v_i \in A\ot A$ is called a {\bf solution of the associative Yang-Baxter equation of weight $\lambda$} in $A$ if it satisfies  Eq.~(\mref{eq:AYBE}).
\end{enumerate}
\end{defn}

\begin{remark}\mlabel{rem:3ex}
\begin{enumerate}
\item When $\lambda=0$, the equation
\begin{align*}
r_{13}r_{12}-r_{12}r_{23}+r_{23}r_{13}=0
\end{align*}
 is called a {\bf homogeneous associative Yang-Baxter equation},
initiated by Aguiar~\mcite{MA} and further studied in~\mcite{BGN12}.

\item When $\lambda \neq 0$, the equation
\begin{align}
r_{13}r_{12}-r_{12}r_{23}+r_{23}r_{13}=\lambda r_{13}, \mlabel{eq:nhAYBE}
\end{align}
is called a {\bf non-homogeneous associative Yang-Baxter equation}.

\item When $\lambda=-1$, the equation
\begin{align*}
r_{13}r_{12}-r_{12}r_{23}+r_{23}r_{13}=-r_{13}
\end{align*}
is called a {\bf modified associative Yang-Baxter equation} studied by Ebrahimi-Fard~\mcite{EF02}.
\end{enumerate}
\end{remark}

We now give some basic definitions and notations  that will be used in this section. We refer to~\cite[Section~5]{MA}
for the classical results in the case of $\lambda=0$.

\begin{defn} Let $\lambda$ be a given element of $\bfk$, and let $A$ be a unitary algebra and $W$ an $A$-bimodule.
\begin{enumerate}
\item A linear map $\Delta:A\rightarrow A\ot A$ is called a {\bf derivation of weight $\lambda$} if it satisfies Eq.~(\mref{eq:cocycle}), that is,
\begin{align*}
    \Delta(ab) =a\cdot \Delta(b)+\Delta(a)\cdot b+\lambda(a\ot b) \, \text{ for }\, a, b\in A.
    \end{align*}

\item A linear map $\Delta_{r} :A\rightarrow W$ associated to an element $r\in W$ is called {\bf principal} if it satisfies
    \begin{align}
    \Delta_{r}(a) :=a\cdot r- r\cdot a-\lambda (a\ot 1) \, \text{ for }\, a\in A. \mlabel{eq:prin}
    \end{align}

\item An element $r\in W$ is called {\bf $A$-invariant} if it satisfies
\begin{align*}
a\cdot r=r\cdot a \, \text{ for }\, a\in A.
\end{align*}
\end{enumerate}
\end{defn}

For an algebra $A$, $A\ot A\ot A$ is viewed as an $A$-bimodule via
\begin{align*}
a\cdot (b\ot c\ot d)=ab\ot c\ot d \, \text{ and }\, (b\ot c\ot d)\cdot a =b\ot c\ot da,
\end{align*}
where $a, b, c, d\in A$.

\begin{prop}
The $\Delta_{r}$ defined by Eq.~(\mref{eq:prin}) is a derivation of weight $\lambda$. \mlabel{prop:weid}
\end{prop}

\begin{proof}
For $a, b \in A$, it follows from Eq.~(\mref{eq:prin}) that
\begin{align*}
a\cdot \Delta_{r}(b)+\Delta_{r}(a)\cdot b=&a\cdot \Big(b\cdot r-r\cdot b-\lambda (b\ot 1)\Big)+\Big(a\cdot r-r\cdot a-\lambda (a\ot 1)\Big)\cdot b\\
=&ab\cdot r-a\cdot r\cdot b -\lambda ab\ot 1 +a\cdot r \cdot b-r\cdot ab-\lambda (a\ot b)\\
=&ab\cdot r-r\cdot ab-\lambda ab\ot 1-\lambda (a\ot b)\\
=&\Delta_{r}(ab)-\lambda (a\ot b),
\end{align*}
as desired.
\end{proof}

The following result plays a crucial role to obtain a relationship between solutions of  weighted AYBEs and weighted $\epsilon$-unitary bialgebras.

\begin{theorem}
Let $A$ be a unitary algebra and $r=\sum_i u_i \ot v_i\in A\ot A$.
Then the principle derivation of weight $\lambda$
$$\Delta_r: A\rightarrow A\ot A,\quad a\mapsto a\cdot r - r\cdot a - \lambda (a\ot 1)$$
is coassociative if and only if the element
\begin{align*}
r_{13}r_{12}-r_{12}r_{23}+r_{23}r_{13}-\lambda r_{13}\in A\ot A\ot A
\end{align*}
is $A$-invariant.
\mlabel{thm:iff}
\end{theorem}

\begin{proof}
On the one hand,
\begin{align*}
(\Delta_{r}\ot \id)\circ \Delta_{r}(a)=&(\Delta_{r}\ot \id) (a\cdot r-r\cdot a-\lambda a\ot 1)\quad (\text{by Eq.~(\ref{eq:prin})})\\
=&(\Delta_{r}\ot \id)\left(\sum_{i}\Big(au_i\ot v_i-u_i\ot v_ia\Big)-\lambda a\ot 1\right)\\
=&\sum_{i}\Delta_{r}(au_i)\ot v_i-\sum_{i}\Delta_{r}(u_i)\ot v_ia-\lambda \Delta_{r}(a)\ot 1\\
=&\sum_{i}\Big(au_i\cdot r-r\cdot au_i-\lambda au_i\ot 1\Big)\ot v_i-\sum_{i}\Big(u_i\cdot r-r\cdot u_i-\lambda u_i\ot 1\Big)\ot v_ia\\
&- \lambda(a\cdot r-r\cdot a-\lambda a\ot 1) \ot 1 \quad (\text{by Eq.~(\ref{eq:prin})})\\
=&\sum_{i, j}\Big(au_iu_j\ot v_j\ot v_i-u_j\ot v_jau_i\ot v_i\Big)-\lambda \sum_{i}au_i\ot 1\ot v_i\\
&-\sum_{i, j}\Big(u_iu_j\ot v_j\ot v_ia-u_j\ot v_ju_i\ot v_ia\Big)+\lambda \sum_{i}u_i\ot 1\ot v_ia\\
&-\lambda \sum_{i}\Big(au_i\ot v_i\ot 1- u_i\ot v_ia \ot 1\Big)+\lambda^2(a\ot 1\ot 1)\\
=&a\cdot r_{13}r_{12}-r_{12}(1\ot a\ot 1)r_{23}-\lambda a\cdot r_{13}-r_{13}r_{12}\cdot a+r_{12}r_{23}\cdot a+\lambda r_{13}\cdot a\\
&-\lambda a\cdot r_{12}+\lambda r_{12} (1\ot a\ot 1)+\lambda^2(a\ot 1\ot 1).
\end{align*}
On the other hand,
\begin{align*}
(\id\ot \Delta_{r})\circ \Delta_{r}(a)=&(\id\ot \Delta_{r}) (a\cdot r-r\cdot a-\lambda a\ot 1)\quad (\text{by Eq.~(\ref{eq:prin})})\\
=&( \id\ot \Delta_{r})\left(\sum_{i}\Big(au_i\ot v_i-u_i\ot v_ia\Big)-\lambda a\ot 1\right)\\
=&\sum_{i}au_i\ot\Delta_{r}( v_i)-\sum_{i}u_i\ot \Delta_{r}(v_ia)-\lambda a\ot \Delta_{r}(1)\\
=&\sum_{i}au_i\ot \Big(v_i\cdot r-r\cdot v_i-\lambda v_i\ot 1\Big)-\sum_{i}u_i\ot \Big(v_i a\cdot r-r\cdot v_i a-\lambda v_i a\ot 1\Big)\\
&+\lambda^2 a\ot 1\ot 1 \quad (\text{by Eq.~(\ref{eq:prin})})\\
=&\sum_{i, j}\Big(au_i\ot v_iu_j\ot v_j-au_i\ot u_j\ot v_jv_i\Big)-\lambda \sum_{i}au_i\ot v_i \ot 1\\
&-\sum_{i, j}\Big(u_i\ot v_iau_j\ot v_j-u_i\ot u_j\ot v_jv_ia\Big)+\lambda \sum_{i}u_i\ot v_ia\ot 1+\lambda^2 (a\ot 1\ot 1)\\
=&a\cdot r_{12}r_{23}-a\cdot r_{23}r_{13}-\lambda a\cdot r_{12}-r_{12}(1\ot a\ot 1)r_{23}+r_{23}r_{13}\cdot a+\lambda r_{12}(1\ot a\ot 1)\\
&+\lambda^2 (a\ot 1\ot 1).
\end{align*}
Comparing the two hands, $\Delta_r$ is coassociative if and only if
\begin{align*}
a\cdot r_{13}r_{12}-\lambda a\cdot r_{13}-r_{13}r_{12}\cdot a+r_{12}r_{23}\cdot a+\lambda r_{13}\cdot a
=a\cdot r_{12}r_{23}-a\cdot r_{23}r_{13}+r_{23}r_{13}\cdot a
\end{align*}
if and only if
\begin{align*}
a\cdot (r_{13}r_{12}- r_{12}r_{23}+ r_{23}r_{13}-\lambda r_{13})
=(r_{13}r_{12}-r_{12}r_{23}+r_{23}r_{13}-\lambda r_{13})\cdot a
\end{align*}
if and only if
$$r_{13}r_{12}-r_{12}r_{23}+r_{23}r_{13}-\lambda r_{13}$$
is $A$-invariant.
\end{proof}

Now we arrive at our main result in this subsection.

\begin{theorem}
Let $(A, m, 1)$ be a unitary algebra and $r\in A\ot A$.
If $r$ is a solution of an AYBE of weight $\lambda$ in $A$, then the quadruple $(A, m, 1, \Delta_r)$ is an $\epsilon$-unitary bialgebra of weight $\lambda$.
\mlabel{thm:mainiff}
\end{theorem}

\begin{proof}
Suppose that $r$  is a solution of an AYBE of weight $\lambda$ in $A$.
By Theorem~\mref{thm:iff}, $\Delta_r: A \to A\ot A$ is coassociative. Further by Proposition~~\mref{prop:weid},
\begin{align*}
    \Delta_r(ab) =a\cdot \Delta_r(b)+\Delta_r (a)\cdot b+\lambda(a\ot b) \, \text{ for }\, a, b\in A.
    \end{align*}
Thus the result holds by Definition~\mref{def:iub}.
\end{proof}

\subsection{A solution of a homogeneous AYBE for matrix algebras}
In this subsection, we give a solution of a homogeneous AYBE for matrix algebras.
\begin{theorem}
Let $L$ be a matrix of $M_n(\bfk)$ such that $L^2=0$. Then $r=L\ot L$ is a solution of a homogeneous AYBE for matrix algebras.
\mlabel{thm:solu}
\end{theorem}

\begin{proof}
Since $L^2=0$ and $r=L\ot L$, we have
\begin{align*}
r_{13}r_{12}&=(L\ot 1\ot L)(L\ot L \ot 1)=L^2\ot L\ot L=0,\\
r_{12}r_{23}&=(L\ot L \ot 1)(1\ot L \ot L)=L\ot L^2\ot L=0,\\
r_{23}r_{13}&=(1\ot L\ot L)(L\ot 1\ot L)=L\ot L\ot L^2=0,
\end{align*}
and so
\begin{align*}
r_{13}r_{12}-r_{12}r_{23}+r_{23}r_{13}=0.
\end{align*}
This completes the proof.
\end{proof}

The following result gives another way to construct an $\epsilon$-unitary bialgebra of weight zero on the matrix algebra $M_{n}(\bfk)$ obtained in Subsection~\mref{sec:sub}.

\begin{coro}
Let $L\in M_n(\bfk)$ such that $L^2=0$ and $r=L\ot L\in M_n(\bfk)\ot M_n(\bfk)$.  Define the linear map $\Delta_r: M_n(\bfk)\rightarrow M_n(\bfk)\ot M_n(\bfk)$ by setting
\begin{align*}
\Delta_r(M):=M\cdot r-r\cdot M \, \text{ for }\, M\in M_n(\bfk).
\end{align*}
 Then the quadruple $(M_n(\bfk), \frakm, E,\col)$ is an $\epsilon$-unitary bialgebra of weight zero.
\mlabel{coro:iub}
\end{coro}
\begin{proof}
It follows from Theorems~\mref{thm:mainiff} and \mref{thm:solu}.
\end{proof}

\subsection{Weighted AYBEs and Rota-Baxter operators}
In this subsection, we derive a Rota-Baxter operator of weight $-\lambda$  from an AYBE of weight $\lambda$, which generalizes the result studied in~\mcite{Agu00}. See also~\cite[Theorem~1.3]{BGN12}.

\begin{defn}\cite{Bax60, Gub}
Let $A$ be a unitary algebra and $\lambda$ a given element of $\bfk$. A linear operator $P: A\rightarrow A$ is called a {\bf Rota-Baxter operator of weight $\lambda$} if it satisfies the Rota-Baxter equation
\begin{align}
P(a)P(b)=P(aP(b))+P(P(a)b)+\lambda P(ab)\, \text{ for }\, a,b\in A.
\mlabel{eq:RBal}
\end{align}
Then  the pair $(A, P)$ is called a {\bf Rota-Baxter algebra of weight $\lambda$}.
\end{defn}

The following result captures the relation between an AYBE of weight $\lambda$ and a Rota-Baxter operator of weight $-\lambda$.
\begin{theorem}
Let $r=\sum_{i}u_i\ot v_i$ be a solution of an  AYBE of weight $\lambda$ in $A$. Then the linear operator
\begin{align}
P_r: A\rightarrow A, \quad a\mapsto P_r(a):=\sum_{i}u_i av_i
\mlabel{eq:RBidd}
\end{align}
is a Rota-Baxter operator of weight $-\lambda$.
\mlabel{thm:RB2}
\end{theorem}
\begin{proof}
Define a $\mathbf{k}$-trilinear map

$$ p: A \times A \times A \to     A,\quad   (u,v, w)
  \mapsto uxvyw \, \text{ for }\, u, v, w, x, y \in A, $$
which induces a $\mathbf{k}$-linear map
$$
      h
  \colon  A \ot A \ot A
  \to     A,
  \quad   u \ot v \ot w
  \mapsto uxvyw.
$$
Since $r=\sum_{i}u_i\ot v_i$ is a solution of a  AYBE of weigh $\lambda$,  we have
\begin{align*}
r_{13}r_{12}-r_{12}r_{23}+r_{23}r_{13}=\lambda r_{13},
\end{align*}
that is,
\begin{align}
\sum_{i,j}u_iu_j\ot v_j\ot v_i-\sum_{i,j}u_i\ot v_iu_j\ot v_j+\sum_{i,j}u_j\ot u_i\ot v_iv_j=\lambda \sum_{i}u_i\ot 1\ot v_i.\label{eq:AYBE2}
\end{align}
Applying $h$ on both sides of Eq.~(\mref{eq:AYBE2}),
\begin{align*}
\sum_{i,j}u_iu_jxv_jy v_i-\sum_{i,j}u_ix v_iu_jy v_j+\sum_{i,j}u_jx u_iy v_iv_j=\lambda \sum_{i}u_ix y v_i.
\end{align*}
Thus
\begin{align*}
P_r(P_r(x)y)-P_r(x)P_r(y)+P_r(xP_r(y))=\lambda P_r(xy),
\end{align*}
that is,
\begin{align*}
P_r(x)P_r(y)=P_r(P_r(x)y)+P_r(xP_r(y))-\lambda P_r(xy),
\end{align*}
which implies that $P_r$ is a  Rota-Baxter operator of weight $-\lambda$.
\end{proof}

\begin{coro}
Let $A$ be a unitary algebra.
\begin{enumerate}
\item ~\mcite{MA}
If $r=\sum_{i}u_i\ot v_i$ be a solution of a homogeneous AYBE of in $A$, then the linear operator
$P: A\rightarrow A$ given by Eq.~(\mref{eq:RBidd})
is a Rota-Baxter operator of weight zero.
\item ~\mcite{EF02}
If $r=\sum_{i}u_i\ot v_i$ be a solution of a modified AYBE of in $A$, then the linear operator
$P: A\rightarrow A$ given by Eq.~(\mref{eq:RBidd})
is a Rota-Baxter operator of weight $-1$.
\end{enumerate}
\end{coro}

\begin{proof}
It follows from Theorem~\mref{thm:RB2} by taking $\lambda=0$ and $\lambda=-1$, respectively.
\end{proof}

\subsection{A bijection between the solutions of weighted AYBEs and Rota-Baxter operators}
In this subsection, we shall give a bijection between the solutions of weighted AYBEs and Rota-Baxter operators on matrix algebra $M_{n}(\bfk)$, which generalize the results studied in~\cite[Section 3]{Gub18}.

Let $E_{i j}\in M_{n}(\bfk)$, $1\leq i, j\leq n$, be the matrix whose entry in the $i$-th row, $j$-th column is $1$,
and zero in all other entries. Note that $E_{ij}E_{kl}=\delta_{jk}E_{il}$ and $E_{i j}$, $1\leq i, j\leq n$, are
a linear basis of $M_{n}(\bfk)$. Recall that $P_r$ is defined in Eq.~(\mref{eq:RBidd}).

\begin{theorem}
Let $r$ be a solution of an  AYBE of weight $\lambda$ in $M_{n}(\bfk)$. Then the map $r\rightarrow P_r$ is a bijection between the set of the solutions of AYBE of weight $\lambda$ in $M_{n}(\bfk)$ and the set of Rota-Baxter operators of weight $-\lambda$ on $M_{n}(\bfk)$.
\mlabel{thm:bijection}
\end{theorem}

\begin{proof}
On the one hand, a linear operator
\begin{align}P:M_{n}(\bfk)\rightarrow M_{n}(\bfk),\quad E_{pq} \mapsto\sum_{i,l}t_{ip}^{ql}E_{il}, \,\text{ where }\, t_{ip}^{ql}\in \bfk \mlabel{eq:RB3}.
\end{align}
is a Rota-Baxter operator of weight -$\lambda$ on $M_{n}(\bfk)$ if and only if
\begin{align}
P(E_{ab}) P(E_{cd}) = P(P(E_{ab})E_{cd}) + P(E_{ab}P(E_{cd})) - \lambda P(E_{ab}E_{cd})\,\text{ for } \, E_{ab}, E_{cd} \in M_{n}(\bfk).
\mlabel{eq:rbpe}
\end{align}
Further, it follows from Eq.~(\mref{eq:RB3}) that
\begin{align*}
P(E_{ab})P(E_{cd})=\sum_{i,j}t_{ia}^{bj}E_{ij}\sum_{k,l}t_{kc}^{dl}E_{kl}=\sum_{i,j,l}t_{ia}^{bj}t_{jc}^{dl}E_{il}.
\end{align*}
Similarly,
\begin{align*}
P(P(E_{ab})E_{cd}) =\sum_{i,j,l}t_{ja}^{bc}t_{ij}^{dl}E_{il},\quad
P(E_{ab}P(E_{cd})) =\sum_{i,j,l}t_{bc}^{dj}t_{ia}^{jl}E_{il}
\end{align*}
and
\begin{align*}
P(E_{ab}E_{cd}) =
\begin{cases}
P(E_{ad})=\sum_{i,l}t_{ia}^{dl}E_{il} &\text{ if } b=c,\\
0 &\text{ if } b\neq c.
\end{cases}
\end{align*}
We have two cases to consider.

\noindent{\bf Case 1.} $b=c$. In this case, Eq.~(\mref{eq:rbpe}) holds if and only if
\begin{align}
\sum_{j}\left(t_{ia}^{bj}t_{jb}^{dl}-t_{ja}^{bb}t_{ij}^{dl}-t_{bb}^{dj}t_{ia}^{jl}\right)=-\lambda t_{ia}^{dl}.
\mlabel{eq:RBid2}
\end{align}

On the other hand, an element $$r=\sum_{i,j,k,l}t_{ij}^{kl}E_{ij}\ot E_{kl} \in M_{n}(\bfk)\ot M_{n}(\bfk)$$
is a solution of the AYBE of weight $\lambda$ if and only if
\begin{align}
r_{13}r_{12}-r_{12}r_{23}+r_{23}r_{13}=\lambda r_{13} \mlabel{eq:AYBEE}
\end{align}
where
\begin{align*}
r_{12}=\sum_{i,j,k,l}t_{ij}^{kl}E_{ij}\ot E_{kl}\ot 1, r_{13}=\sum_{i,j,k,l}t_{ij}^{kl}E_{ij}\ot 1 \ot E_{kl}
\, \text{ and }\,
r_{23}=\sum_{i,j,k,l}t_{ij}^{kl}1 \ot E_{ij}\ot E_{kl}.
\end{align*}
Now we have
\begin{align}\label{eq:r1312}
r_{13}r_{12}&=\sum_{i,j,k,l,p,s,t}t_{pi}^{st}t_{ij}^{kl}E_{pj}\ot E_{kl}\ot E_{st}=\sum_{i,j,k,l,p,s,t}t_{ij}^{kl}t_{pi}^{st}E_{pj}\ot E_{kl}\ot E_{st},\\ \nonumber
&=\sum_{i,j,k,l,p,s,t}t_{pj}^{kl}t_{ip}^{st}E_{ij}\ot E_{kl}\ot E_{st}\quad (\text{ by exchanging the index $i$ and $p$}).
\end{align}
Similarly,
\begin{align}
r_{12}r_{23}=\sum_{i,j,k,l,p,s,t}t_{ij}^{kl}t_{lp}^{st}E_{ij}\ot E_{kp}\ot E_{st}\, \text{ and }\,
r_{23}r_{13}=\sum_{i,j,k,l,p,q,t}t_{ij}^{kl}t_{pq}^{st}E_{pq}\ot E_{ij}\ot E_{kt}.
\mlabel{eq:r1223}
\end{align}
By substituting the summands from Eqs.~(\mref{eq:r1312}) and ~(\mref{eq:r1223}) in Eq.~(\mref{eq:AYBEE}) and gathering the coefficient of the tensor $E_{ij}\ot E_{kk}\ot E_{st}$, we obtain
\begin{align}
\sum_{p}\left(t_{ij}^{lp}t_{pl}^{st}-t_{pj}^{ll}t_{ip}^{st}-t_{ll}^{sp}t_{ij}^{pt}\right)=-\lambda t_{ij}^{st}.
\mlabel{eq:RBid3}
\end{align}
So Eq.~(\mref{eq:AYBEE}) is equivalent to Eq.~(\mref{eq:RBid3}).
Note that Eq.~(\mref{eq:RBid2}) and Eq.~(\mref{eq:RBid3}) coincide up to exchange of index by the following permutation
\begin{align*}
\left(
  \begin{array}{cccc}
    a & b & d & l \\
    j & l & s & t \\
  \end{array}
\right).
\end{align*}
In summary, a linear operator $P$ is a Rota-Baxter operator of weight $-\lambda$ on $M_n(\bfk)$ if and only if Eq.~(\mref{eq:RBid2}) holds.
An element $r$ is a solution of the AYBE of weight $\lambda$ if and only if Eq.~(\mref{eq:RBid3}) is valid.
Eq.~(\mref{eq:RBid2}) is equivalent to Eq.~(\mref{eq:RBid3}). So in this case,
there is a bijection, namely $\phi$, from the set of the solutions of AYBE of weight $\lambda$ in $M_{n}(\bfk)$
to the set of Rota-Baxter operators of weight $-\lambda$ on $M_{n}(\bfk)$.

\noindent{\bf Case 2.} $b\neq c$. In this case, $P(E_{ab}E_{cd}) = 0$ and we can only consider $P$ is a Rota-Baxter operator of weight zero on $M_n(\bfk)$.
Then, with the condition $\lambda= 0$,  Eq.~(\mref{eq:rbpe}) holds if and only if
\begin{align}
\sum_{j}\left(t_{ia}^{bj}t_{jc}^{dl}-t_{ja}^{bc}t_{ij}^{dl}-t_{bc}^{dj}t_{ia}^{jl}\right)=0.
\mlabel{eq:RBid1}
\end{align}

On the other hand, an element
$$r=\sum_{i,j,k,l}t_{ij}^{kl}E_{ij}\ot E_{kl} \in M_n(\bfk) \ot M_n(\bfk)$$ is a solution of the AYBE of weight zero in $M_{n}(\bfk)$
if and only if
\begin{align}
r_{13}r_{12}-r_{12}r_{23}+r_{23}r_{13}=0 \mlabel{eq:AYBEEE}
\end{align}
Substituting the summands from Eqs.~(\mref{eq:r1312}) and ~(\mref{eq:r1223}) in Eq.~(\mref{eq:AYBEEE}) and gathering the coefficient of the tensor $E_{ij}\ot E_{kl}\ot E_{st}$, we obtain
\begin{align}
\sum_{p}\left(t_{ij}^{kp}t_{pl}^{st}-t_{pj}^{kl}t_{ip}^{st}-t_{kl}^{sp}t_{ij}^{pt}\right)=0.
\mlabel{eq:RBid4}
\end{align}
Observe that Eq.~(\mref{eq:RBid1}) and Eq.~(\mref{eq:RBid4}) coincide up to the following permutation of indexes
\begin{align*}
\left(
  \begin{array}{ccccc}
    a & b &c& d & l \\
    j & l &l& s & t \\
  \end{array}
\right).
\end{align*}
So in this case, there is also a bijection $\phi$ from the set of the solutions of AYBE of weight zero in $M_{n}(\bfk)$
to the set of Rota-Baxter operators of weight zero on $M_{n}(\bfk)$.

Finally, the map $\phi$ acts as
\begin{align*}
\phi\left(r=\sum_{i,j,k,l}t_{ij}^{kl}E_{ij}\ot E_{kl}\right)=P \, \text{ such that }\, P(E_{pq})=\sum_{i,l}t_{ip}^{ql}E_{il}=\sum_{i,j,k,l}t_{ij}^{kl}E_{ij}E_{pq} E_{kl}=P_r(E_{pq}).
\end{align*}
So the linear operator $P$ is exactly the linear operator $P_r$ defined in Eq.~(\mref{eq:RBidd}).
Therefore, the bijection $\phi$ is precisely the map $r\rightarrow P_r$. This completes the proof.
\end{proof}

\begin{exam}
Consider the matrix algebra $M_{2}(\mathbb{C})$. Aguiar~\cite[Example 5.4.5]{MA} showed that all nonzero solutions of AYBE of weight zero in $M_{2}(\mathbb{C})$ are
\begin{align*}
r_1 = E_{12}\ot E_{12}, \, r_2 =E_{22}\ot E_{12},\, r_3 =(E_{11}+E_{22})\ot E_{12}\, \text{ and } \, r_4 =E_{11}\ot E_{12}-E_{12}\ot E_{11},
\end{align*}
up to conjugation, transpose and scalar multiple.
By Theorem~\mref{thm:bijection}, all nonzero Rota-Baxter operators (viewed as matrices) of weight zero in $M_{2}(\mathbb{C})$ up to conjugation, transpose and scalar multiple are the following:
\begin{enumerate}
\item $P_{r_1}(E_{21})=E_{12}$,  $P_{r_1}(E_{11})=P_{r_1}(E_{12})=P_{r_1}(E_{22})=0$;
\item $P_{r_2}(E_{21})=E_{22}$,  $P_{r_2}(E_{11})=P_{r_2}(E_{12})=P_{r_2}(E_{22})=0$;
\item $P_{r_3}(E_{11})=E_{12}$, $P_{r_3}(E_{21})=E_{22}$, $P_{r_3}(E_{12})= P_{r_3}(E_{22})=0$;
\item $P_{r_4}(E_{11})=E_{12}$, $P_{r_4}(E_{21})=-E_{11}$, $P_{r_4}(E_{12})=P_{r_4}(E_{22})=0$.
\end{enumerate}
See also in~\cite{BGP, Gub18, TZS14}.

\end{exam}

\section{Weighted quasitriangular $\epsilon$-unitary bialgebras and dendriform algebras}\label{sec:qua}
In this section, we introduce the concept of weighted quasitriangular $\epsilon$-unitary bialgebras, which generalize the quasitriangular $\epsilon$-bialgebras  initiated by Aguiar~\mcite{MA}. We show that any weighted quasitriangular $\epsilon$-unitary bialgebra can be made into a dendriform algebra.

\subsection{Weighted quasitriangular $\epsilon$-unitary bialgebras}
\begin{defn}
Let $(A, m, 1)$ be a unitary algebra. A {\bf quasitriangular infinitesimal unitary bialgebra} (abbreviated {\bf quasitriangular $\epsilon$-unitary bialgebra}) {\bf of weight $\lambda$} is a quadruple $(A, m, 1, r)$ consisting of a unitary algebra  $(A, m, 1)$ and a solution $r\in A\ot A$  of an associative Yang-Baxter equation of weight $\lambda$.
\mlabel{def:qua}
\end{defn}

Recall that $\Delta_{r}$ is defined in Eq.~(\mref{eq:prin}) for a $r\in A\ot A$.

\begin{remark}
\begin{enumerate}
\item By Theorem~\mref{thm:mainiff}, the quadruple $(A, m, 1,\Delta_r)$ is indeed an $\epsilon$-unitary bialgebra of weight $\lambda$.

\item A quasitriangular $\epsilon$-bialgebra studied in~\mcite{MA} is a quasitriangular $\epsilon$-unitary bialgebra of weight zero. In this case, the quadruple $(A, m, 1, \Delta_r)$ is  an $\epsilon$-unitary bialgebra of weight zero.
\end{enumerate}
\end{remark}

Strongly motivated by Aguiar~\mcite{MA}, we record some properties of weighted quasitriangular $\epsilon$-unitary bialgebras.

\begin{prop}
Let $(A, m, 1,  r)$ be a quasitriangular $\epsilon$-unitary bialgebra of weight $\lambda$ and $\Delta :=\Delta_r$. Then
\begin{align}\mlabel{eq:a}
&\Delta(a)=a\cdot r-r\cdot a-\lambda (a\ot 1),\\ \mlabel{eq:b}
&(\Delta\ot \id )(r)=-r_{23}r_{13}, \, \text{ and }\, \\ \mlabel{eq:c}
&(\id \ot \Delta)(r)=r_{13}r_{12}+\lambda (r_{13}-r_{12}).
\end{align}
Conversely, if an $\epsilon$-unitary bialgebra $(A, m, 1,\Delta)$ of weight $\lambda$ satisfies Eqs.~(\mref{eq:a}), (\mref{eq:b}) and (\mref{eq:c})
for some $r = \sum_i u_i\ot v_i \in A\ot A$, then $(A, m, 1, r)$ is a quasitriangular $\epsilon$-unitary bialgebra of weight $\lambda$ and $\Delta=\Delta_r$.
\mlabel{prop:quaiff}
\end{prop}

\begin{proof}
Eq.~(\mref{eq:a}) follows directly from Eq.~(\mref{eq:prin}). Having Eq.~(\mref{eq:a}) in hand, we have
\begin{align*}
(\Delta\ot \id )(r)=&\sum_{i}\Delta(u_i)\ot v_i= \sum_{i}\Big(u_i\cdot r-r\cdot u_i-\lambda u_i\ot 1\Big)\ot v_i\\
=&\sum_{i, j}\Big(u_iu_j\ot v_j\ot v_i-u_j\ot v_ju_i\ot v_i\Big)-\lambda\sum_{i}u_i\ot 1\ot v_i\\
=&r_{13}r_{12}-r_{12}r_{23}-\lambda r_{13}=-r_{23}r_{13}  \quad (\text{by Eq.~(\ref{eq:AYBE})}).
\end{align*}
With a similar argument,
\begin{align*}
(\id \ot \Delta)(r)=&\sum_{i}u_i\ot \Delta(v_i)= \sum_{i} u_i \ot\Big(v_i\cdot r-r\cdot v_i-\lambda v_i\ot 1\Big) \\
=&\sum_{i, j}\Big (u_i\ot v_iu_j\ot v_j-u_i\ot u_j\ot v_j v_i\Big)-\lambda \sum_{i}u_i\ot v_i\ot 1\\
=&r_{12}r_{23}-r_{23}r_{13}-\lambda r_{12}=r_{13}r_{12}+\lambda (r_{13}-r_{12})\quad (\text{by Eq.~(\ref{eq:AYBE})}).
\end{align*}
Conversely, if an $\epsilon$-unitary bialgebra $(A, m, 1, \Delta)$ of weight $\lambda$ satisfies Eqs.~(\mref{eq:a}), (\mref{eq:b}) and (\mref{eq:c}) for some $r\in A\ot A$, then the same calculation shows that $r$ is a solution of an AYBE of weight $\lambda$, and so $(A, m, 1, r)$ is a quasitriangular $\epsilon$-bialgebra of weight $\lambda$. We note finally that $A\ot A$ is an $A$-bimodule by Eq.~(\mref{eq:dota}), then $\Delta=\Delta_r$ follows from Eqs.~(\mref{eq:prin}) and~(\mref{eq:a}).
\end{proof}

\subsection{Dendriform algebras from weighted quasitriangular $\epsilon$-unitary bialgebras}
In this subsection, we derive a dendriform algebra from a weighted quasitriangular $\epsilon$-bialgebra.

\begin{defn}\mcite{LR01}
A {\bf dendriform algebra} is a $\bfk$-module $D$ together with two binary operations $\prec:D\ot D \rightarrow D$ and $\succ: D\ot D \rightarrow D$ that satisfy the following relations:
\begin{align*}
(a\prec b)\prec c=&a\prec (b\prec c+b\succ c),\\
(a\succ b)\prec c=&a\succ (b\prec c),\\
a\succ (b\succ c)=&(a\prec b+a\succ b)\succ c \, \text{ for }\, a, b,c \in D .
\end{align*}
\end{defn}

We record the following lemma as a preparation.

\begin{lemma}\mcite{EF02, EG05}
Let $(A, P)$ be a Rota-Baxter algebra of weight $\lambda$. Define two binary
operations $\succ, \prec$ on $A$ by
\begin{align*}
a\succ b:=P(a)b\, \text{ and }\, a\prec b:=aP(b)+\lambda ab \, \text{ for }\, a, b\in A.
\end{align*}
Then the triple $(A, \succ, \prec )$ is a dendriform algebra.
\mlabel{lem:dend}
\end{lemma}

The following result captures a relationship between weighted quasitriangular $\epsilon$-unitary bialgebras  and dendriform algebras.
\begin{theorem}
Let $(A, m, 1, r)$ be a quasitriangular $\epsilon$-unitary bialgebra of weight $\lambda$ with $r=\sum_{i}u_i\ot v_i\in A\ot A$. Define  two binary operations $\succ, \prec$ on $A$ by
\begin{align*}
a\succ b:=\sum_{i}u_iav_ib \, \text { and }\, a\prec b:=\sum_{i}au_ibv_i-\lambda ab.
\end{align*}
Then the triple $(A, \succ, \prec )$ is a dendriform algebra.
\mlabel{thm:quadri}
\end{theorem}

\begin{proof}
The result follows from Theorem~\mref{thm:RB2} and Lemma~\mref{lem:dend}.
\end{proof}

\smallskip

\noindent {\bf Acknowledgments}:
This work was supported by the National Natural Science Foundation
of China (Grant No.\@ 11771191 and 11501267), Fundamental Research Funds for the Central
Universities (Grant No.\@ lzujbky-2017-162), the Natural Science Foundation of Gansu Province (Grant
No.\@ 17JR5RA175).

\end{document}